\documentclass{article}
\usepackage[utf8]{inputenc}
\usepackage[a4paper,top=4.5cm,bottom=4.5cm,left=2.5cm,right=2.5cm]{geometry}
\usepackage{changepage}

\usepackage{amsmath}
\usepackage{amsfonts} 
\usepackage{dsfont}
\usepackage{bm}
\usepackage{amsthm}
\usepackage{amssymb}
\usepackage{xcolor}

\DeclareMathOperator*{\spn}{span}
\newcommand{\muv}{\boldsymbol{\mu}}

\usepackage{algorithm}
\usepackage{algorithmic}

\newtheorem{lemma}{Lemma}[section]

\newtheorem{definition}{Definition}[section]

\theoremstyle{remark}

\usepackage[colorlinks=true,linkcolor=black,anchorcolor=black,citecolor=black,filecolor=black,menucolor=black,runcolor=black,urlcolor=black]{hyperref} 
\usepackage{cleveref}
\usepackage{verbatim}

\usepackage{graphicx}
\usepackage[labelsep=period, labelfont=bf]{caption}
\usepackage{subcaption}

\title{PTPI-DL-ROMs: pre-trained physics-informed deep learning- based reduced order models for nonlinear parametrized PDEs}% in small data regimes}
\author{Simone Brivio, Stefania Fresca, Andrea Manzoni}
\date{}
\begin{document}

\maketitle

\begin{abstract}
    Among several recently proposed data-driven Reduced Order Models (ROMs), the coupling of Proper Orthogonal Decomposition (POD) and deep learning-based ROMs (DL-ROMs) has proved to be a successful strategy to construct non-intrusive, highly accurate, surrogates for the \textit{real time} solution of parametric nonlinear time-dependent PDEs. Inexpensive to evaluate, POD-DL-ROMs are also relatively fast to train, thanks to their limited complexity. However, POD-DL-ROMs account for the physical laws governing the problem at hand only through the training data, that are usually obtained through a full order model (FOM) relying on a high-fidelity discretization of the underlying equations. Moreover, the accuracy of POD-DL-ROMs strongly depends on the amount of available data. In this paper, we consider a major extension of POD-DL-ROMs by enforcing the fulfillment of the governing physical laws in the training process -- that is, by making them physics-informed  --  to compensate for possible scarce and/or unavailable data and improve the overall reliability. To do that, we first complement POD-DL-ROMs with a \textit{trunk net} architecture, endowing them with %the mesh-agnostic property, i.e.
    the ability to compute the problem's solution at every point in the spatial domain, and ultimately enabling a seamless computation of the physics-based loss by means of the strong continuous formulation. Then, we introduce %aiming to deal with complex nonlinear problems, we develop
    an efficient training strategy that limits the notorious computational burden entailed by a physics-informed training phase. In particular, we take advantage of the few available data to develop a low-cost pre-training procedure; %, aiming at enabling a fast and steady convergence to suitable minima in the loss landscape; 
    then, we fine-tune the architecture in order to further improve the prediction reliability. 
    Accuracy and efficiency of the resulting pre-trained physics-informed DL-ROMs (PTPI-DL-ROMs) are then assessed on a set of test cases ranging from non-affinely parametrized advection-diffusion-reaction equations, to nonlinear problems like the Navier-Stokes equations for fluid flows.
\end{abstract}

\section{Introduction}

Mathematical models expressed in terms of nonlinear partial differential equations (PDEs) are ubiquitous in Applied Sciences and Engineering, enabling scientists to describe complex patterns or phenomena arising in a wealth of contexts, including Fluid Dynamics, Structural Mechanics, and Computational Biology, to mention a few examples. Despite providing accurate approximations to these problems, high-fidelity, full-order models (FOMs) are often computationally prohibitive. For instance, finite element-based FOMs usually entail spatial discretizations of the problem domain regulated by very small mesh sizes $h>0$, thus yielding algebraic problems whose dimension $N_h$ can easily reach millions of unknowns  scaling with the number of vertices $\{\mathbf{x}_i\}_{i=1}^{N_h}$. The computational burden is even higher in a parametric setting, where the problem's solution also depends on a vector $\muv \in \mathcal{P} \subset \mathbb{R}^{N_h}$ of input parameters, in addition to the spatial coordinates and, possibly, the time variable $t \in \mathcal{T} = [0,T]$ -- with $T>0$. In several situations, we wish to explore the entire solution manifold $\mathcal{S}_{N_h} = \{\mathbf{u}(\muv,t) = [u(\mathbf{x}_i,\muv,t)]_{i=1}^{N_h} \in \mathbb{R}^{N_h}: (\muv,t) \in \mathcal{P} \times \mathcal{T}\}$ to address either real-time simulations or multi-query problems (like, e.g., uncertainty quantification or parameter estimation) -- two tasks that are usually out of reach with high-fidelity FOMs. Reduced Order Models (ROMs) are usually employed to accomplish these tasks,  replacing the high-fidelity problem by one featuring a much lower numerical complexity \cite{benner2017model,benner2020model_1,benner2020model_2}. Among several available options, the reduced basis (RB) method \cite{QMN_16} exploits the $\boldsymbol{\mu}$-dependence of the solution manifold to generate a reduced subspace that approximates the FOM solution, and an algebraic system to solve through a (Galerkin-type) projection of the FOM onto the reduced subspace. Low-dimensional linear subspaces can be obtained, e.g., through proper orthogonal decomposition (POD). For parametrized PDEs, POD selects the most relevant modes by computing the singular value decomposition of a matrix collecting a set of FOM snapshots. However, despite the mathematical rigor of this idea, assembling the resulting ROMs might be an extremely intrusive task. Moreover, to ensure accuracy, the dimension of the ROM can increase remarkably, making computational benefits negligible in several applications involving nonlinear, time-dependent problems \cite{farhat2020}. Nonetheless, a dynamical system --  even if of small size -- has to be solved at the reduced level, in case of time-dependent problems. All these issues make classical ROM strategies unfeasible in view of real-time predictions. 
%Thus, it is inevitable the development of suitable ROM techniques \cite{QMN_16, farhat2020}, which delegate the majority of the computational burden to an \textit{offline} training phase, while offering a slim and quick \textit{online} evaluation phase. 

Recently, %Deep Learning algorithms -- and 
Deep Neural Networks (DNNs) %in particular -- 
have become extremely popular in reduced order modeling, aiming at {\em (i)} surrogating the expensive construction of ROMs by exploiting, e.g., feedforward NNs \cite{hesthaven2018,wang2019non,bhattacharya2021}, recurrent NNs \cite{wang2020recurrent} and ResNets architectures \cite{oleary2022}, or {\em (ii)} performing dimensionality reduction through, e.g., convolutional autoencoders \cite{gonzalez2018deep,san2019artificial,lee2020model}, 
or even {\em (iii)} 
coupling these two latter tasks, e.g., through  deep learning-based reduced order models (DL-ROMs) \cite{franco2023deep, fresca2021dlrom, mucke2021, pant2021}, among which we mention POD-DL-ROMs \cite{fresca2022poddlrom, fresca2021fluids, fresca2022microstructures,brivio2023error}, POD-LSTM-ROMs \cite{fresca2023long}, and fast SVD-ML-ROMs  \cite{drakoulas2023fastsvd} -- indeed very similar to the former. These latter strategies -- and, more generally speaking, data-driven ROMs -- are becoming more and more ubiquitous among ROM techniques, especially when dealing with nonlinear time-dependent parametrized PDEs, as they offer a striking accuracy, excellent generalization capabilities, and an extremely fast \textit{online} inference phase. However, the reliability of deep learning-based strategies strongly depends not only on the proper design of neural network architectures, but also on the amount of collected samples the neural network is fed with during the \textit{offline} training phase. In fact, snapshots' collection plays a major role in the \textit{offline} phase of DL-ROMs -- in addition to the training of DNN architectures. When the FOM is particularly hard to solve, %because of the high amount of involved degrees of freedom (dofs), 
the required computational time might also limit the amount of collected snapshots. On the other hand, it is well known that data scarcity represents a serious possible concern about the overall reliability of DNNs-based methods.  \\ %Thus, the affordability of the FOM employed to generate the aforementioned data samples constitute a major bottleneck in the technique pipeline.

Recent advances in Scientific Machine Learning (SciML) suggest to compensate small data regimes with an optimization phase driven by a physics-informed loss formulation, employing the \textit{a priori} knowledge of the physical model governing the problem at hand \cite{zhu2023reliable}. 
Indeed, physics-informed techniques -- such as, e.g., physics-informed neural networks \cite{raissi2019physics,karniadakis2021physics} or physics-informed deep operator networks \cite{wang2021learning} -- can foster the accuracy remarkably, even in contexts where data are entirely missing in a region of the time-parameter domain $\mathcal{P} \times \mathcal{T}$.
It is worth to notice that enforcing physics-based constraints has proven to be successful even when dealing with a sufficient amount of data, implying in these cases a remarkable improvement of the overall reliability of the considered method \cite{chen2021}. See, e.g., \cite{yang2021b,hernandez2021deep,hernandez2021structure} for more general approaches to build physics-aware, data-driven ROMs. 

Even though physics-based constraints shall enhance the training phase of almost any DL-ROM as long as the underlying physical model is available, in practice a suitable minimum in the physics-informed loss function is hard to achieve, thus requiring a large number of epochs in the neural network training \cite{wang2021learning,raissi2019physics}, or even causing convergence failure in the optimization phase \cite{daw2023mitigating}. Approaches exploiting an interplay between data and physics  \cite{chen2021} often reduce the risk of optimization failure and limit the number of epochs required to achieve a suitable accuracy. Nevertheless, the computation of a physics-informed loss term requires a substantial amount of computational time and resources per epoch: indeed, it usually relies on automatic differentiation, and entails residuals' evaluation -- that is, the repeated assembling of FOM arrays for parametrized problems -- with larger and larger difficulties as the order of the required derivatives and the amount of nonlinearities increases. Thus, it is essential to develop a framework for the residual computation which only mildly impacts on the required training time and resources.

Within this paper, we extend the proposed deep learning-based ROMs enhanced by POD (POD-DL-ROMs) \cite{fresca2022poddlrom,fresca2021fluids, fresca2022microstructures, brivio2023error} to make them physics-aware, thus formulating a pre-trained physics-informed DL-ROM (PTPI-DL-ROM) paradigm. Our main contributions are thus the following ones:
\begin{itemize}
    \item we design a physics-based loss functional that enforces the governing physical laws in the training process, thus compensating for a shortage of available supervised data;
    \item we devise a fast and efficient pre-training strategy that blends the few available data and the underlying physics to speed up the convergence to a suitable minima in the proposed loss function;
    \item we further enhance the proposed architecture through a fine-tuning strategy, which is also characterized from an abstract point of view.
\end{itemize}
Then, we showcase the potential of our approach through small data interpolation and extrapolation tasks. In particular, we emphasize the complexity of extrapolation tasks, where the available labeled data belong to a small region of the time-parameter domain $\mathcal{P} \times \mathcal{T}$, thus leaving the most part of this latter substantially uncovered by training data, and making the prediction task harder. Specifically, we demonstrate that a suitably trained PTPI-DL-ROM is endowed with excellent extrapolation capabilities, up to $400\%$ with respect to both the parameter range and the time interval from which supervised data are sampled.

The structure of the paper is as follows. 
In Section 2, we state the general formulation of the problem we deal with, and we specify the family of neural network architectures we intend to focus on. In Section 3 we provide a general framework on the integration of physics to compensate for the drawbacks entailed by small data regimes. In Section 4, we propose the PTPI-DL-ROM paradigm, which intertwines information extracted from the underlying physical problem and the contribution of the available labeled data. We emphasize that this novel paradigm is supplied with an efficient training algorithm consisting in a pre-training phase and a fine-tuning stage. Section 5 is then devoted to the assessment of the proposed framework through a series of numerical experiments, including also nonlinear and non-affine differential problems. Finally, the concluding Section 6 reports a brief discussion on possible further developments and some open issues.

\section{A general framework: low-rank DL-based architectures}
\label{sec:mathematical-overview}

Within this section we aim at presenting the theoretical framework this work relies on. Specifically, after introducing the general class of parametrized differential problems we focus on, we describe the architecture and the properties of low-rank deep learning based ROMs; among the latter, we distinguish spatially discretized (or mesh-based) approaches -- such as POD+DNN \cite{hesthaven2018}, POD-DL-ROMs \cite{fresca2022poddlrom}  and POD-DeepONets \cite{lu2022comprehensive} -- from spatially continuous (or mesh-agnostic) techniques -- such as DeepONets \cite{lu2021learning}, characterizing them in a data-driven context.

\subsection{Problem formulation}
We aim at treating generic nonlinear time-dependent parametric PDEs whose general formulation can be expressed as follows:
\begin{equation}
    \label{eq:general_formulation}
    \left\{
    \begin{aligned}
        \frac{\partial u}{\partial t} + \mathcal{A}(\muv,t)u(\muv,t) + \mathcal{N}(u(\muv,t),\muv,t) &= f(\muv,t) , &\mbox{in } \Omega \times (0,T] \\
        \mathcal{B}(\muv)u(\muv,t) &= g(\muv,t) , &\mbox{on } \partial{\Omega} \times (0,T]\\
        u(\muv,0) &= u_0(\muv), & \mbox{in } \Omega.
    \end{aligned}
    \right. 
\end{equation}
The problem is set in the spatial (compact and bounded) domain $\Omega \subset \mathbb{R}^d$, $d \geq 1$. The solution $u = u(\mathbf{x},\muv,t)$ at any point $\mathbf{x} \in \Omega$ depends on both the input parameter vector $\muv \in \mathcal{P} \subset \mathbb{R}^{p}$, with $\mathcal{P}$ compact, and the time variable $t \in \mathcal{T} = [0,T]$, for some $T > 0$. In the formulation \eqref{eq:general_formulation}, $\mathcal{A}$ represents a linear (second-order, elliptic) operator, $\mathcal{N}$ denotes a nonlinear operator, while  $\mathcal{B}$ is a generic boundary operator, encoding e.g. the trace of a function on $\partial \Omega$ or the normal flux, or their linear combination, to represent Dirichlet, Neumann, or Robin conditions, respectively; finally, $u_0 = u_0(\muv)$ denotes the initial datum.

Once discretized through the finite element method (FEM) on a mesh with step size $h>0$, problem \eqref{eq:general_formulation} yields a nonlinear dynamical system of size $N_{FOM}$ -- this latter might coincide with the number $N_h$ of vertices $\{\mathbf{x}_i\}_{i=1}^{N_h}$ appearing in the finite element discretization mesh \cite{quarteroni2017}, even if thi choice is not restrictive. We thus define $V_h = \spn(\{\psi_i\}_{i=1}^{N_{FOM}})$ as the finite dimensional subspace involved in the problem discretization, where $\psi_1, \ldots, \psi_{N_{FOM}}$ denote the finite element basis functions. Note that also the choice of the FEM as high-fidelity technique is not restrictive,  other options being also possible, such as, e.g., the Spectral Element Method, or the Finite Volume Methods.  Without loss of generality, the general formulation of the high-fidelity FOM problem reads as
\begin{equation}
    \label{eq:general_formulation_FOM}
    \left\{
    \begin{aligned}
        {\bf M}(\muv) \frac{\partial {\bf u}_{FOM}}{\partial t}(\muv,t) + {\bf A}(\muv,t){\bf u}_{FOM} (\muv,t) + {\bf N}({\bf u}_{FOM}(\muv,t),\muv,t) &= {\bf f}(\muv,t), & t \in (0,T] \\
        {\bf u}_{FOM}(\muv, 0) &= {\bf u}_{0}(\muv), & 
    \end{aligned}
    \right. 
\end{equation}
where ${\bf M}(\muv) \in \mathbb{R}^{N_{FOM} \times N_{FOM}}$ denotes the parameter-dependent (symmetric, positive definite) mass matrix, while the (possibly, time-dependent) stiffness matrix ${\bf A}(\muv, t) \in \mathbb{R}^{N_{FOM} \times N_{FOM}}$ and the nonlinear term ${\bf N}(\cdot, \muv, t) : \mathbb{R}^{N_{FOM}} \rightarrow \mathbb{R}^{N_{FOM}}$ are the discrete counterparts of the linear operator $\mathcal{A}$ %and $\mathcal{B}$, 
and of the nonlinear operator $\mathcal{N}$, respectively. The remaining right hand side contributions are encompassed by $\mathbf{f} := \mathbf{f}(\muv,t) \in \mathbb{R}^{N_{FOM}}$, while ${\bf u}_{0}(\muv)$ denotes the initial datum.  We remark that the FOM discrete solution $\mathbf{u}_{FOM}(\muv, t) = \{u_{FOM,i} (\muv, t) \}_{i=1}^{N_{FOM}}$ can be directly linked to the function approximating the PDE solution, that is, $u(\mathbf{x}, \muv, t) \approx u_{FOM}(\mathbf{x}, \muv, t) = \sum_{i =1}^{N_{FOM}} u_{FOM,i}(\muv, t) \psi_i(\mathbf{x})$. 

We emphasize that we normally collect the evaluation of the high-fidelity solution at the mesh vertices, namely
$\mathbf{u}_h(\muv,t) = \{u_{FOM}(\mathbf{x}_i, \muv, t)\}_{i=1}^{N_h}$, or we choose $\mathbf{u}_h(\muv,t) = \mathbf{u}_{FOM}(\muv,t)$; for the sake of readability, hereon we consider $N_{FOM} = N_h$, without loss of generality. 
Then, we employ the high-fidelity solver to collect the labeled training dataset for varying values of $(\muv,t) \in \mathcal{P}_{sup} \times \mathcal{T}_{sup} \subseteq \mathcal{P} \times \mathcal{T}$ ($\mathcal{P}_{sup},\mathcal{T}_{sup}$ compacts). Formally:
\begin{itemize}
    \item we select a finite set of time-parameter values through a sampling strategy, namely,
        \begin{equation*}
            \mathsf{P}_{sup} = \{\mu_j \in \mathcal{P}_{sup}\}_{j=1}^{N_s}, \qquad  \mathsf{T}_{sup} = \{k|\mathcal{T}_{sup}|/N_t\}_{k=1}^{N_t};
        \end{equation*}
    \item then, we generate a set of snapshots with the high-fidelity solver \eqref{eq:general_formulation_FOM}, thus obtaining the supervised (labeled) training dataset
    \begin{equation*}
         \mathsf{D}_{sup} = \{(\muv,t, \mathbf{u}_h(\muv,t)) \mbox{, for any } (\muv,t) \in \mathsf{P}_{sup} \times \mathsf{T}_{sup} \}.
    \end{equation*}
\end{itemize}
We stress that the computational burden entailed by the FOM solver might become unaffordable in the case of complex problems, so that we might be able to collect only a handful of data. Within the present work, we focus on two different sampling strategies for the synthetic data generation phase -- namely, we assume either {\em (i)} to sample scattered data in the entire time-parameter domain, thus leaving large portions of it substantially uncovered by training data, or {\em (ii)} to draw solution snapshots from a small region of the entire time-parameter space, that is $\mathcal{P}_{sup} \times \mathcal{T}_{sup} \subsetneq \mathcal{P} \times \mathcal{T}$, where $|\mathcal{P}_{sup} \times \mathcal{T}_{sup}| \ll |\mathcal{P} \times \mathcal{T}|$. 
We employ the same procedure to collect a suitably representative test dataset to evaluate the performance of the ROM strategy: thus, we set $\mathcal{P}_{test} \times \mathcal{T}_{test} \equiv \mathcal{P} \times \mathcal{T}$ and generate $\mathsf{D}_{test}$, corresponding to the test instances $\mathsf{P}_{test} \times \mathsf{T}_{test} \subset \mathcal{P}_{test} \times \mathcal{T}_{test}$. We remark that the test dataset is needed only \textit{a posteriori} to test the ROM accuracy, assuming to manage to query the computationally expensive high-fidelity solver a sufficiently large number of times ($|\mathsf{D}_{test}| \gg |\mathsf{D}_{sup}|$).

\subsection{Neural network architectures and low-rank decompositions}
\label{sec:nn-architectures}

Within the present work we consider an operator learning task that involves the approximation of mappings between (possibly infinite-dimensional) function spaces \cite{kovachi2023neural}. Specifically, assuming that $\mathcal{X}$ is a compact set in either $\mathbb{N}$ or $\mathbb{R}^d$, our purpose is to approximate the parametric mapping $u: \mathcal{P} \times \mathcal{T} \rightarrow L^2(\mathcal{X})$. In the following, we focus in particular on low-rank deep learning-based approaches as they feature a powerful combination of linear dimensionality reduction (entailed by a low-rank decomposition) and nonlinear approximation  capabilities due to a neural network (NN) core \cite{goodfellow2016deep,han2018solving}: indeed, they have been widely used for learning parametric operators even in real-world applications involving, e.g., chemical combustion \cite{kumar2024combustion}, seismology \cite{haghighat2024endeeponet}, electrophysiology \cite{fresca2021electrophys} and micro-mechanical systems \cite{fresca2022microstructures}. In the following, we formally define low-rank deep learning-based architectures.

\begin{definition}
\label{def:low-rank-deep-learning-based-approch}
Let $\mathcal{X}$ be a compact set of either $\mathbb{N}$ or $\mathbb{R}^d$ and denote by $\mathcal{P} \times \mathcal{T} \subset \mathbb{R}^{p+1}$ the time-parameter domain. Assume that $u: \mathcal{P} \times \mathcal{T} \rightarrow L^2(\mathcal{X})$ is the parametric map to be approximated, and that $\{\hat{v}_k\}_{k=1}^{N}: \mathcal{X} \rightarrow \mathbb{R}^N$ is a collection of $N < + \infty$ (possibly learnable) basis functions spanning a low-dimensional subspace of $L^2(\mathcal{X})$. We call {\em low-rank deep learning-based architecture} any NN architecture 
\[
\hat{u}: \mathcal{P} \times \mathcal{T} \rightarrow L^2(\mathcal{X}), \qquad  \hat{u}(\muv,t): = \sum_{k=1}^N \hat{v}_k \hat{q}_k(\muv,t) \approx u(\muv,t),
\]
where the branch network $\hat{\mathbf{q}} = \{\hat{q}_k\}_{k=1}^N: \mathcal{P} \times \mathcal{T} \rightarrow \mathbb{R}^N$ is any (deep) feedforward neural network. 
\end{definition}

This definition allows us to cast -- apparently different -- strategies like, e.g., space-discrete POD+DNN architectures (like, e.g., POD-DL-ROMs), or space-continuous DeepONets in a common framework, ultimately combining them to devise a new, physics-informed POD-DL-ROM strategy. Indeed, we emphasize that in our framework the distinction between space-discrete and space-continuous approaches resides in the choice of the compact set $\mathcal{X}$, as we are going to detail in the following.

\subsubsection*{Space-discrete DL-based architecture: POD+DNN}

If we suppose that $\mathcal{X} \equiv \{1,\ldots,N_h\}$, then $\hat{v}_k : \mathcal{X}  \rightarrow \mathbb{R}^N$ is usually represented using a matrix $\mathbf{V} \in \mathbb{R}^{N_h \times N}$ defined as $\mathbf{V}_{ik} = \hat{v}_k(i)$, $k=1,\ldots,N$, $i=1,\ldots,N_h$. Thus, we obtain 
\[
\mathbb{R}^{N_h} \ni \hat{\mathbf{u}}_h(\muv,t) = \mathbf{V}\hat{\mathbf{q}}(\muv,t) \approx \mathbf{u}_h(\muv,t)
\]
and, since the present approach is associated to the numerical discretization, we deem it {\em space-discrete} or {\em mesh-based}. We stress that the peculiar lightweight training of {\em space-discrete} approaches is due to a remarkable reduction of the computational burden of the optimization phase: indeed, the global spatial basis functions collected in $\mathbf{V}$ are pre-computed before the training of the branch neural network $\hat{\mathbf{q}}$. This usually ensures lower backpropagation costs and fosters stability during the optimization phase. 

Hereon we focus on POD-based strategies \cite{willcox2002balanced, QMN_16, farhat2020, buithanh2003proper}, thus delving into the POD+DNNs family of architectures \cite{hesthaven2018}; however, we remark that in the literature there are some alternatives, see e.g. \cite{oleary2022}. For the sake of brevity, we do not delve deeper into density results and error estimates involving low-rank deep learning-based architectures: the interested reader can refer to, e.g., \cite{lanthaler2022error,brivio2023error}. Building a POD+DNN architecture entails the following steps: % are characterized by an offline-online splitting of the computational procedure, similarly to common ROM techniques for differential problems. In particular,  during  the offline phase of POD+DNNs:
\begin{itemize}
    \item [(i)] from the supervised training dataset $\mathsf{D}_{sup}$, we extract the snapshots matrix as follows 
    \begin{equation*}
        \mathbf{U} \in \mathbb{R}^{N_h \times N_sN_t} \mbox{ s.t. } \mathbf{U}[:,(k+(j-1)N_s)] = \mathbf{u}_h(\muv_j,t_k) \mbox{ for } j=1,\ldots,N_s \mbox{ and } k = 1,\ldots,N_t;
    \end{equation*}
    \item [(ii)] we perform a preliminary dimensionality reduction aiming at compressing the available data, heuristically assuming that the parametric dependence and the dynamics can be resolved by few coordinates. In this respect, we  employ singular value decomposition (SVD) \cite{eckart1936} -- or its randomized version \cite{szlam2014} -- and retain the first $N \le N_h$ modes. Formally, we seek to factorize the snapshot matrix $\mathbf{U} \approx \mathbf{V} \bm{\Sigma} \mathbf{W}^T$, where $\mathbf{V} \in \mathbb{R}^{N_h \times N}$ and $\mathbf{W} \in \mathbb{R}^{N \times N_{data}}$ are two semi-orthogonal matrices, while $\bm{\Sigma} = \textnormal{diag}(\{\sigma_k\}_{k=1}^N) \in \mathbb{R}^{N \times N}$ is a diagonal matrix whose entries $\sigma_k$, for $k=1,\ldots,N$, are referred to as the singular values. It is well known \cite{QMN_16,eckart1936,chen2021} that (thin) SVD entails the best $N$-rank linear decomposition, in the sense that,
    \begin{equation*}
        \mathbf{V} = \arg \min_{\mathbf{W} \in \mathbb{R}^{N_h \times N}: \mathbf{W}^T\mathbf{W} = \mathbf{I}} \frac{|\mathcal{P}_{sup} \times \mathcal{T}_{sup}|}{N_s N_t} \sum_{j=1}^{N_s} \sum_{k=1}^{N_t} \|\mathbf{u}_h(\muv,j,t_k) - \mathbf{W}\mathbf{W}^T\mathbf{u}_h(\muv_j,t_k)\|^2;
    \end{equation*}
    moreover, the (squared) error entailed by the preliminary POD dimensionality reduction amounts to
    \begin{equation*}
        \frac{|\mathcal{P}_{sup} \times \mathcal{T}_{sup}|}{N_s N_t} \sum_{j=1}^{N_s N_t} \|\mathbf{u}_h(\muv_j, t_k) - \mathbf{V}\mathbf{V}^T\mathbf{u}_h(\muv_j,t_k)\|^2 = \sum_{k>N} \sigma_k^2.
    \end{equation*}
    We then compress the available snapshots, namely, $\mathbf{Q} = \mathbf{V}^T \mathbf{U} \in \mathbb{R}^{N \times N_sN_t} = \{\{\mathbf{q}(\muv_j,t_k)\}_{k=1}^{N_t}\}_{j=1}^{N_s}$. 
    
\item  [(iii)] We train the branch network, i.e., we optimize the POD+DNN's weights and biases with respect to the data-driven loss functional
\begin{equation}\label{eq:vanilla_pod_dnn}
    \mathcal{L}_{POD+DNN} = \sum_{j=1}^{N_s} \sum_{k=1}^{N_t} \|\hat{\mathbf{q}}(\muv_j, t_k) - \mathbf{q}(\muv_j, t_k)\|^2
\end{equation}
employing suitable optimization algorithms, such as Adam \cite{kingma2017adam}.
\end{itemize}

\begin{figure}[ht]
\centering
\includegraphics[width=0.975\textwidth]{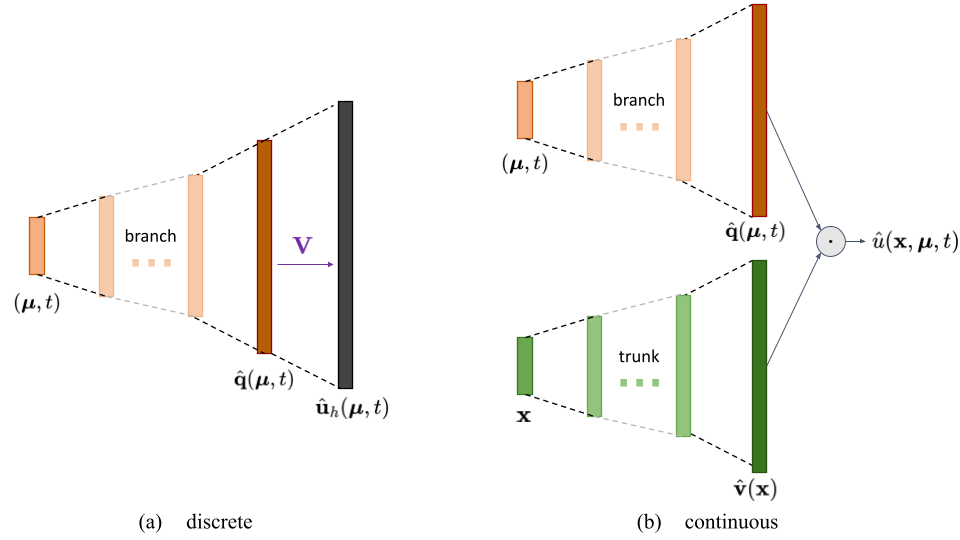}
\caption{Example of {\em (a)} space-discrete (POD+DNN) and {\em (b)} space-continuous (DeepONets) low-rank DL-based architectures.}
\label{fig:low-rank-schemes}
\end{figure}

\subsubsection*{Space-continuous DL-based architecture: DeepONets}

Despite they are ubiquitous in the operator learning framework because of their outstanding approximation and generalization capabilities, POD+DNNs architectures are limited by their discrete nature, in the sense that they cannot be evaluated continuously in the spatial domain $\Omega$. {\em Space-continuous} (or {\em mesh-agnostic}) low-rank DL-based architectures aim at filling this deficiency by casting the operator learning task in a suitable infinite-dimensional setting. Formally, we choose $\mathcal{X} \equiv \Omega$ (in the notation of Definition \ref{def:low-rank-deep-learning-based-approch}), so that $\hat{v}_k \in L^2(\Omega)$ for any $k=1,\ldots,N$.
DeepONets are a particular instance of space-continuous low-rank DL-based architectures that model the basis function $\hat{\mathbf{v}} = \{\hat{v}_k\} \in \mathbb{R}^N$ through a neural network, which is referred to as the \textit{trunk net}. In the end, by combining the trunk net and the branch net, we obtain the mesh-agnostic approximator 
\[
\hat{u}(\mathbf{x},\muv,t) := \hat{\mathbf{v}}(\mathbf{x}) \cdot \hat{\mathbf{q}}(\muv,t) \approx u(\mathbf{x},\muv,t), \qquad \forall \mathbf{x} \in \overline{\Omega}, \, \muv \in \mathcal{P}, t \in \mathcal{T}.
\]
We stress that, in practice, one recovers further regularity (at least $\hat{v}_k \in \mathcal{C}(\overline{\Omega})$) depending on the smoothness of the activation function of the trunk net \cite{lanthaler2022error}: the continuity assumption is essential for the mesh-agnostic evaluation of the approximated solution $\hat{u}(\mathbf{x},\muv,t)$. 

Nevertheless, we emphasize that the mesh-agnostic property of DeepONets comes with some drawbacks. Indeed, if POD+DNNs entail the pre-computation of the POD basis $\mathbf{V} \in \mathbb{R}^{N_h \times N}$, within the DeepONet paradigm the trunk net is trained simultaneously with the branch net. Specifically, their data-driven loss functional is
\begin{equation}
\label{eq:loss_deeponet_supervised}
    \mathcal{L}_{DeepONet} = \sum_{i=1}^{N_h} \sum_{j=1}^{N_s} \sum_{k=1}^{N_t} |u_{h,i}(\muv_j,t_j) - \hat{\mathbf{v}}(\mathbf{x}_i) \cdot \hat{\mathbf{q}}(\muv_j,t_k)|^2,
\end{equation}
where $u_{h,i}$ is the $i$-th element of the high-fidelity solution $\mathbf{u}_h$, usually  minimized by employing either Adam \cite{kingma2017adam} or other adequate optimization algorithms.
We remark that the simultaneous training of branch and trunk implies in a loss of training stability and entails higher backpropagation costs, resulting in a more computationally demanding training. Nonetheless, the discrete basis functions provided by POD are optimal for the training data at hand, while we have no guarantee of optimality for the trained trunk net's basis functions. Moreover, POD modes are ordered depending on the ``energy" that they retain, whereas the trunk net functions are not sorted. We refer the reader to Fig.~\ref{fig:low-rank-schemes} for a visual comparison of the architectures of space-discrete and space-continuous approaches.

\subsection{POD-DL-ROMs as low-rank deep learning-based architectures}

Among low-rank DL-based architectures, we now focus on the POD-DL-ROM paradigm, which can be seen as a special space-discrete, low-rank approach, suitable for small data regimes. %After characterizing its properties, the underlying neural network architecture, and its loss, we motivate  their relevancy in the field and why they are indicated for .
POD-DL-ROMs can be seen as a particular instance of POD+DNNs, employing an autoencoder-based architecture as underlying neural network; for an in-depth analysis of the technique we refer the reader to \cite{brivio2023error, fresca2022poddlrom}). Aside from the preprocessing procedure through POD, POD-DL-ROMs (see Fig.~\ref{fig:pod-dl-rom}) are obtained by combining an encoder $\Psi'(\bm{\theta}_{\Psi'}): \mathbb{R}^N \rightarrow \mathbb{R}^n$ that further compresses the information coming from the POD coefficients to $p + 1\le n < N$ latent variables, a decoder $\Psi(\bm{\theta}_{\Psi}): \mathbb{R}^n \rightarrow \mathbb{R}^N$, and a (deep) feedforward neural network $\phi(\bm{\theta}_{\phi}): \mathbb{R}^{p+1} \rightarrow \mathbb{R}^n$ that maps the input parameters to the aforementioned latent representation to perform both a change of variables and an augmentation of the input parameters. We denote by $\bm{\theta}_{\mathbf{q}} := (\bm{\theta}_{\Psi'}, \bm{\theta}_{\Psi}, \bm{\theta}_{\phi})$ the vector including all the trainable weights and biases of the aforementioned neural networks; hereon, we omit them in the notation for the sake of clarity.

In their original formulation, POD-DL-ROMs rely on a fully data-driven optimization phase. Thus, the underlying physics governing the problem at hand is only inferred from the collected snapshots' data. In other words, no physics-based constraints are imposed neither in the loss formulation, nor as hard constraint in the architecture. Thus, without including any physics-based knowledge, the data-driven loss function to be minimized during the training phase is made by two distinct terms:  
\begin{equation}
    \label{eq:POD-DL-ROM-loss}
     \mathcal{L}_{POD-DL-ROM} = \mathcal{L}_{POD-DL-ROM}(\omega_N, \omega_n) = \omega_{N} \mathcal{L}_{N} + \omega_n \mathcal{L}_{n},
\end{equation}
where $\omega_{N}, \omega_{n} \ge 0$, with $\omega_{N} +  \omega_{n} = 1$, and
\begin{equation*} 
           \mathcal{L}_{N}  = \sum_{j=1}^{N_s}
         \sum_{k=1}^{N_t} 
         \|(\Psi \circ \phi)(\muv_j,t_k) - \mathbf{V}^T\mathbf{u}_h(\muv_j,t_k)\|^2, \qquad \mathcal{L}_{n} =  
         \sum_{j=1}^{N_s} 
         \sum_{k=1}^{N_t} 
         \|\Psi'(\mathbf{V}^T\mathbf{u}_h(\muv_j,t_k)) - \phi(\muv_j,t_k)\|^2. 
  \end{equation*}
Optimizing the loss $\mathcal{L}_{POD-DL-ROM}$ during the \textit{offline} training phase implies {\em (i)} the minimization of the reconstruction error, represented by $\mathcal{L}_N$, and {\em (ii)} the fulfillment of a suitable latent representation, obtained by enforcing the compatibility condition represented by $\mathcal{L}_n$. We emphasize that the encoder $\Psi'$ is employed only at the training stage to seek a suitable latent encoding, however it does not play any active role in the approximation of the forward map $(\muv,t) \mapsto \mathbf{u}_h(\muv,t)$ since it is not connected to the rest of the architecture: for this reason, we can discard it in the \textit{online} prediction phase. Then, we define as 
\[
\hat{\mathbf{u}}_h(\muv, t) = \mathbf{V} \hat{\mathbf{q}}(\muv, t) = \mathbf{V} (\Psi \circ \phi)(\muv, t) \approx \mathbf{u}(\muv, t)
\]
the ROM network, where $\mathbf{V} \in \mathbb{R}^{N_h \times N}$ is the POD matrix, thus resulting in a special instance of POD+DNN approach. Nonetheless, if we compare the POD-DL-ROM loss function \eqref{eq:POD-DL-ROM-loss} with the data-driven POD+DNN loss function \eqref{eq:vanilla_pod_dnn} we immediately see that they differ by the additive term $\omega_n \mathcal{L}_{n}$. From a computational viewpoint, the latter term allows POD-DL-ROMs to benefit from the autoencoder-based architecture by enforcing an adequate representation in the latent space through a data-driven regularization, thus alleviating the complexity of the decoder \cite{brivio2023error, fresca2022poddlrom}. We emphasize that limiting the complexity by design is a huge advantage when dealing with small data regimes: indeed, if the amount of available training data is insufficient,  over-parametrized architectures may suffer from overfitting. 

\begin{figure}[ht]
\centering
\includegraphics[width=0.75\textwidth]{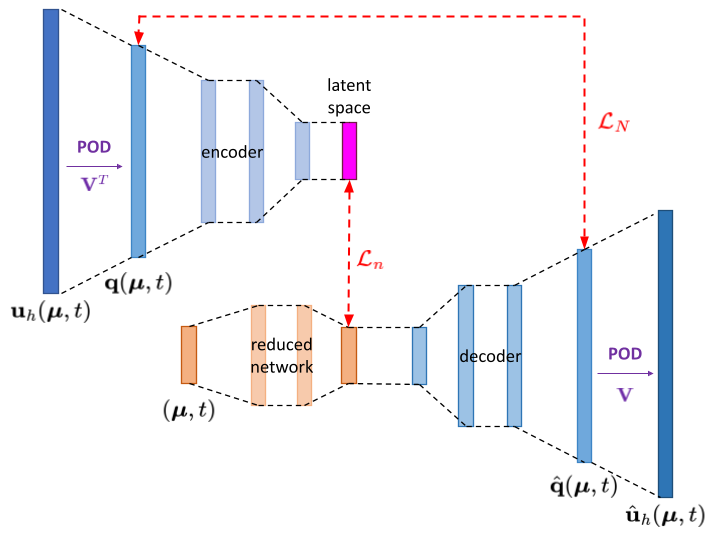}
\caption{The POD-DL-ROM architecture: we highlight the POD projection, the {discrete lifting} induced by the matrices $\mathbf{V}^T$ and $\mathbf{V}$, and the autoencoder-based architecture.}
\label{fig:pod-dl-rom}
\end{figure}

Note that the approaches presented within this section have been originally conceived as data-driven paradigms. However, a data-driven optimization for low-rank DL-based architectures -- as for any other DL architecture -- requires a suitable amount of data to be properly trained in order to be able to generalize well. To  compensate the  availability of (usually, few) training data, recent trends in the literature propose physics-based loss functionals in the optimization phase \cite{wang2021learning,chen2021}. The purpose of this work is to  show how to equip our former POD-DL-ROM architecture with physical constraints, ultimately yielding  our new PTPI-DL-ROM architecture.

\section{Mitigating the shortfalls of small data regimes with physics} %: a general framework}
\label{sec:mitigating-the-shortfalls-small-data}
Several recent works in the literature \cite{hesthaven2018, fresca2022poddlrom,  brivio2023error, lu2021learning, lu2022comprehensive, li2020fourier, mucke2021} showed that data-driven DL-based architectures may require a large amount of supervised training samples depending on the variability of the problem and the required test set accuracy. However, in the case of complex problems, data collection might become expensive or even unaffordable. Hence, training data could be scattered in the time-parameter space, with large portions of this latter substantially uncovered %-- that is, we might deal with a {\em small data} framework 
\cite{xu2023small, brigato2021close}.  
It is well known that the generalization capabilities of data-driven neural network paradigms are skewed in the case of scarce or unavailable labeled data \cite{zhu2023reliable}. Within this section, we analyze in detail data scarcity/unavailability from a theoretical standpoint. 
Formally, following the notation of Section \ref{sec:mathematical-overview}, we do not assume that the time-parameter sampling domain of the supervised training set and the test set coincide. To better clarify our setting, let us define both {\em interpolation} and {\em extrapolation} tasks as follows.
\begin{definition}
    We deal with an interpolation task in the time-parameter domain if $\mathcal{P}_{test} \times \mathcal{T}_{test} \equiv \mathcal{P}_{sup} \times \mathcal{T}_{sup}$; we deal with an extrapolation task if $\mathcal{P}_{test} \times \mathcal{T}_{test} \supsetneq \mathcal{P}_{sup} \times \mathcal{T}_{sup}$.
\end{definition}
In practice, data scarcity and/or unavailability may be due to {\em (i)} a limited number of training samples (small data interpolation regime, which occurs when $\mathcal{P}_{test} \times \mathcal{T}_{test} \equiv \mathcal{P}_{sup} \times \mathcal{T}_{sup}$ and $|\mathsf{D}_{sup}|$ is small), {\em (ii)} the absence of available labeled data in an extended region of the time-parameter space (extrapolation regime), or {\em (iii)} a combination of the former two cases. To mitigate the limits of data-driven architectures in the latter scenarios, physics-informed DL paradigms have arisen as powerful techniques to learn parametric operators without requiring labeled input-output observations \cite{raissi2019physics, wang2021learning}. The hypothesis which they are built upon is that the {terms of the} underlying physical models \eqref{eq:general_formulation} -- and their discretized version \eqref{eq:general_formulation_FOM} as well -- convey the same information of the data synthetically obtained from them through high-fidelity solvers. The main advantage is that the physics-based loss formulation, which is obtained by minimizing the residual of \eqref{eq:general_formulation} -- or equivalently of \eqref{eq:general_formulation_FOM} -- is unsupervised in the sense that it does not require paired input-output observations \cite{safonova2023ten,wang2021learning}. Thanks to this latter observation, we can arbitrarily sample the input parameters from the residual time-parameter space $\mathcal{P}_{res} \times \mathcal{T}_{res}$, enforcing the physics-based constraint instead of minimizing the reconstruction error in the regions of the time-parameter space where training data are not available. Thus, in general, we choose $\mathcal{P}_{res} \times \mathcal{T}_{res} \supseteq \mathcal{P}_{test} \times \mathcal{T}_{test}$, and we define the unlabeled dataset as follows,
\begin{equation*}
    \mathsf{D}_{res} = \biggl\{(\muv,t)_j \in \mathcal{P}_{res} \times \mathcal{T}_{res} \biggr\}_{j=1}^{N_{res}}.
\end{equation*}
We note that we can sample $\mathsf{D}_{res}$ efficiently -- using for instance random samplers or more involved algorithms \cite{mckay2023comparison} -- since the data belonging to $\mathsf{D}_{res}$ are unlabeled, that is, we do not need to use the high-fidelity solver. We refer to Fig.~\ref{fig:datasets} for a sketch of the time-parameter domains $\mathcal{P}_{sup} \times \mathcal{T}_{sup}, \mathcal{P}_{test} \times \mathcal{T}_{test}$ and $\mathcal{P}_{res} \times \mathcal{T}_{res}$, along with their realizations.

\begin{figure}
    \centering
\includegraphics[width=0.45\textwidth]{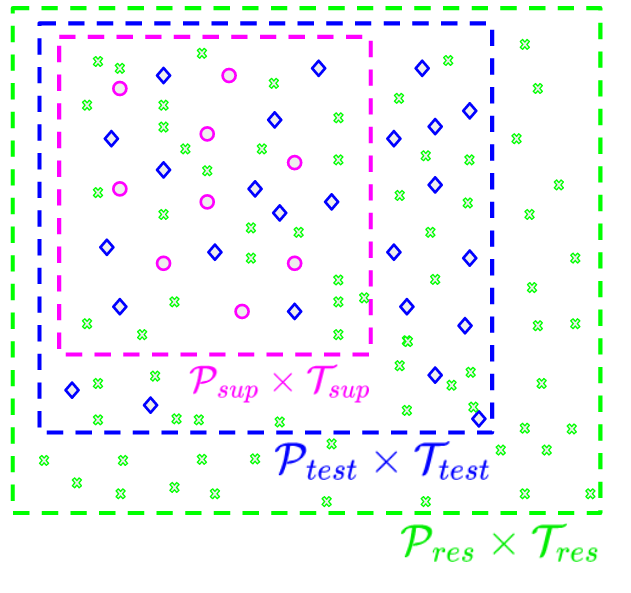}
    \caption{An instance of the time-parameter domains $\mathcal{P}_{sup} \times \mathcal{T}_{sup}, \mathcal{P}_{test} \times \mathcal{T}_{test}$ and $\mathcal{P}_{res} \times \mathcal{T}_{res}$ and their realizations, in dimension $p+1=2$.}
    \label{fig:datasets}
\end{figure}

\subsection{The computation of the physics-based residual}
We now discuss possible alternatives for the computation of  a physics-based residual, identifying the strong continuous formulation as the most versatile among them. 

When dealing with space-discrete low-rank neural network architectures, like POD+DNNs, one of the most natural strategies to compute the physics-based residual could involve the discretized high-fidelity formulation of the problem \eqref{eq:general_formulation_FOM} or its projection on a low-rank subspace. For instance, the technique presented in \cite{chen2021} concerns POD+DNNs entailing the approximation ${\bf u}_h = {\bf u}_h(\muv, t) \approx \mathbf{V} \hat{\mathbf{q}}(\muv,t)$. From the high-fidelity formulation \eqref{eq:general_formulation_FOM}, by employing a Galerkin projection we can write the physics-based residual of the problem at hand as
\begin{equation}
\label{eq:residual_discr}
    \mathcal{R}^{discr}_{PDE}(\muv,t) = {\bf V}^T \biggl({\bf M}(\muv) {\bf V} \frac{\partial \hat{\mathbf{q}}}{\partial t}(\muv,t) + {\bf A}(\muv,t){\bf V} \hat{\mathbf{q}} (\muv,t) + {\bf N}({\bf V} \hat{\mathbf{q}}(\muv,t),\muv,t) - \mathbf{f}(\muv,t)\biggr),  
\end{equation}
if $t>0$, and
\begin{equation*}
    \mathcal{R}^{discr}_{IC}(\muv) = \hat{\mathbf{q}}(\muv, 0) -  {\bf V}^T{\bf u}_{FOM,0}(\muv)
\end{equation*}
otherwise. 
However, this approach may struggle in the presence of, e.g., complex physics involving non-polynomial nonlinearities and non-affine parameter dependencies. Indeed:
\begin{itemize}
    \item the authors of \cite{chen2021} demonstrate how to apply their proposed approach only to linear problems and to equations featuring only quadratic nonlinearities. Indeed, their method cannot be generalized to non-polynomial nonlinearities without additional costly efforts. To avoid the explicit computation of the high-dimensional nonlinear term, namely ${\bf N}({\bf V} \hat{\mathbf{q}}(\muv,t),\muv,t)$, one might still rely on hyper-reduction methods, which may entail a large reduced mesh and remain expensive;
    \item for the computation of the residual as in \eqref{eq:residual_discr}, we still have to assemble the FOM matrices even for samples belonging to the residual dataset, thus implying additional costs both in the data collection phase and in the training stage, especially if the problem depends non-affinely on the parameters. For example, suppose that the stiffness matrix $\mathbf{A}(\muv,t)$ is non-affine, namely, it does not admit a linear decomposition of the form
    \begin{equation*}
        \mathbf{A}(\muv,t) = \sum_p \theta_l(\muv,t) \mathbf{A}_l, \quad \mbox{for } \theta_l : \mathbb{R}^p \mapsto \mathbb{R}, \mathbf{A}_l \in \mathbb{R}^{N_h \times N_h}.
    \end{equation*}
    Then, the proposed algorithm requires to compute and store the set of matrices 
    \begin{equation*}
        \mathcal{A}_{res} = \{\mathbf{V}^T\mathbf{A}(\muv_j,t_j)\mathbf{V} \in \mathbf{R}^{N \times N}: (\muv_j,t_j) \in \mathcal{P}_{res} \times \mathcal{T}_{res}\}_{j=1}^{N_{res}}
    \end{equation*} 
    for the entire duration of the training, thus entailing a computational bottleneck, especially  if $|\mathcal{A}_{res}| = N_{res}\gg 1$;
    \item the residual dataset cannot be resampled without summoning the high-fidelity solver, which is computationally prohibitive in the case of real applications involving non-affine and nonlinear problems. It is worth to notice that resampling techniques (and evolutionary sampling) are actually useful to ensure convergence to adequate minima when dealing with a physics-based loss formulation \cite{daw2023mitigating}.
\end{itemize}
Possible alternatives could rely on, e.g., the weak formulation  of the problem \cite{deryck2022wpinns} or a wavelet-based technique \cite{ernst2022certified}. However, the variational formulation of the physical problem still requires tedious and expensive numerical integrations, and the wavelet-based counterpart is only available  (with periodic wavelets) for simple domains and equispaced discretizations. Therefore, neither of these two options is suitable for complex problems in non-trivial domains. 

In the following, we consider the strong continuous formulation of the residual based on \eqref{eq:general_formulation} \cite{wang2021learning,zhu2023reliable}, whose continuous nature makes it versatile enough for our purposes. However, we stress that the strong continuous formulation of the residual is properly defined only for spatially continuous low-rank DL-based architectures that enable pointwise evaluations for any $\mathbf{x} \in \overline{\Omega}$. By setting $\hat{u} = \hat{u}(\mathbf{x},\muv,t) = \hat{\mathbf{v}}(\mathbf{x}) \cdot \hat{\mathbf{q}}(\muv,t)$, where $ \hat{\mathbf{v}}(\mathbf{x})$ and $\hat{\mathbf{q}}(\muv,t)$ are the trunk and branch network respectively, the strong formulation of the residual reads
\begin{equation*}
\begin{aligned}
    \mathcal{R}^{sc}_{\Omega}(\mathbf{x},\muv,t) &=  \sum_{k=1}^N \biggl\{ \frac{\partial{\hat{q}_k(\muv,t)}}{\partial{t}}\hat{v}_k(\mathbf{x}) +  \mathcal{A}(\mathbf{x},\muv,t) \hat{q}_k(\muv,t) \hat{v}_k(\mathbf{x})+ \\ &\hspace{3cm}+ \mathcal{N}(\hat{q}_k(\muv,t)\hat{v}_k(\mathbf{x}),\mathbf{x},\muv,t)  - f(\mathbf{x},\muv,t) \biggr\}, 
\end{aligned}
\end{equation*}
if $\mathbf{x} \in \Omega$ and $t>0$; on the other hand, if $\mathbf{x} \in \partial \Omega$ and $t > 0$, we have
\begin{equation*}
\mathcal{R}^{sc}_{\partial{\Omega}}(\mathbf{x},\muv,t) = \sum_{k=1}^N \mathcal{B}(\mathbf{x},\muv) \hat{q}_k(\muv,t) \hat{v}_k(\mathbf{x}) - g(\mathbf{x},\muv,t).
\end{equation*}
Similarly, the residual for the initial conditions could be expressed  as
\begin{equation*}
    \mathcal{R}^{sc}_{IC}(\mathbf{x},\muv) = \sum_{k=1}^N \hat{q}_k(\muv,0) \hat{v}_k(\mathbf{x}) -  u(\mathbf{x}, \muv, 0).
\end{equation*}
On a further note, we emphasize that, while the time derivatives can easily be recovered either with Automatic Differentiation (AD) \cite{baydin2017automatic} or finite difference schemes since the time variable is one-dimensional, when dealing with spatial derivatives, AD is the only reliable technique to provide accurate approximations for any domain shape. However, the evaluation of (possibly high-order) spatial derivatives with AD remains the main bottleneck in the residual computation pipeline, especially in the case of complex underlying neural network models. 

In this respect, we remark that within the present work we consider sophisticated physical models that require (very) deep trunk networks that properly model detailed low-scale effects in the global spatial modes to enable a faithful approximation of the ground truth solution. Thus, we deem it appropriate to perform a preliminary numerical experiment showing how the trunk complexity and the derivative order impact on the computational time required by AD. We proceed as follows. We set $\hat{\mathbf{v}}_{(l,w)}: \mathbb{R}^{d} \rightarrow \mathbb{R}^{N}$ as the trunk network endowed with $l-1$ hidden layers consisting of $w$ neurons each. Without loss of generality, we suppose that the input dimension is $d = 3$ and the output dimension is $N = 10$, and we fix the batch size as $b=100$. Then, we measure the computational time required to evaluate the first and second derivative of the trunk network output with respect to its input, in two different scenarios, namely, {\em (i)} as function of the network depth $l \in \{3,4,\ldots,10\}$, fixing the width to $w = 10$, and {\em (ii)} as function of the width $w \in \{5+3k, \mbox{ for } k = 0,\ldots,8\}$, fixing $l = 10$. We repeat the latter experiment $10$ times for each configuration and we report the results in Fig.~\ref{fig:test_ad}. It is rather  evident that the higher the derivative order, the higher the  burden entailed by derivative computation. Moreover, one can notice that, for both first and  second derivatives, the width  of the trunk network does not impact on the computational time required to calculate derivatives. On the other hand, the spatial derivatives' computational time strongly depends on the depth of the trunk network, especially in the case of the second derivative, for which the resulting trend is very steep. 

We conclude by emphasizing that the derivatives' computation is performed during the training phase and is therefore registered in the backpropagation's computational graph, thus causing the training time to grow quickly, especially in the case of deep trunk networks and physics-based residuals including higher-order derivatives.
For this reason, we deem it necessary to devise a suitable algorithm that limits the computational burden of the training phase entailed by the calculation of the physics-based residual.

\begin{figure}[ht!]
    \centering
    \includegraphics[width=0.925\textwidth]{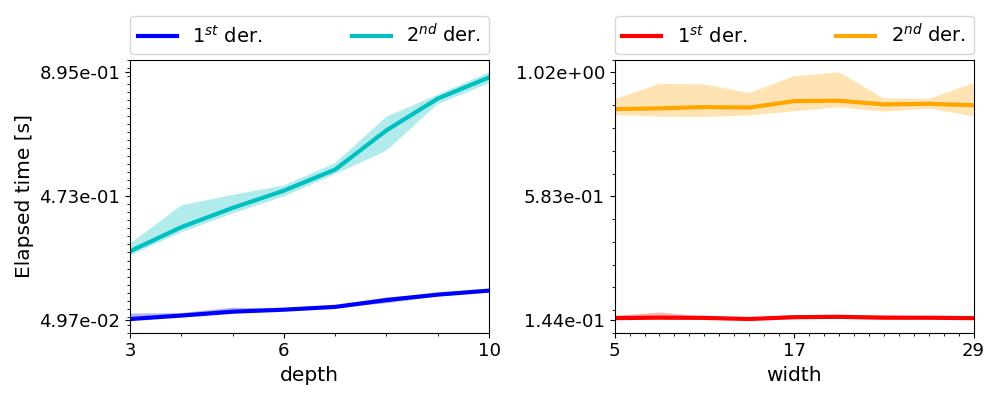}
    \caption{Computational time required for AD as function of the neural network depth {\em (left)} and width {\em (right)}. We remark that the continuous line represent the average value among $10$ runs, whereas the reported interval shows the minimum and the maximum value registered.}
    \label{fig:test_ad}
\end{figure}

\section{The PTPI-DL-ROM paradigm}
In the present section we justify the choice of the underlying neural network architecture at the foundation of our new paradigm. We introduce a suitable loss functional intertwining data-driven and physics-informed contributions, and we describe the full training algorithm, which comprises a lightweight pre-training stage and a fine-tuning phase.

\subsection{Architecture design and pre-training algorithm}

The optimization of DL architectures entailing physics-based loss functions through gradient descent algorithms is a notoriously complex task that shows recurrent modes of failure even in trivial numerical experiments \cite{deryck2023operator, krishnapriyan2021characterizing, wang2022pinns}.
To avoid those optimization failures, in the literature some authors have proposed complex training algorithms based on evolutionary sampling \cite{daw2023mitigating} or adversarial strategies \cite{wang2022l2}. Other works have shown how small amount of data can limit these phenomena \cite{yang2024datadriven, rohrhofer2023role, gopakumar2023loss}. Indeed, combining data and physics can lead to a powerful paradigm for operator learning: if on the one hand physics helps in supplying information to the neural network when data are missing, on the other hand few available data can mitigate the risk of optimization failure and limit the computational burden of training. 

To blend data and physics and devise a quick, stable, and efficient physics-aware training algorithm that we will refer to as PTPI-DL-ROM paradigm, we merge POD-DL-ROMs and DeepONets, leveraging their positive elements. In particular, from POD-DL-ROMs we keep {\em (i)} the preliminar dimensionality reduction through POD and the autoencoder-based architecture in order to extract the most informative features from the handful of available data while avoiding overfitting, and {\em (ii)} the low computational cost of the training procedure, which is characterized by a suitable offline-online decomposition. On the other hand, from DeepONets we retain the characteristic mesh-agnostic property of space-continuous low-rank DL-based architectures that enable the evaluation of the physics-based residual with the strong continuous formulation. Based on this, we devise a lightweight pre-training strategy, proceeding as follows.

\smallskip
\noindent \textbf{Step 1.} The first step of pre-training consists in extracting the most important information from the available data in order to construct a set of representative spatial modes. To do that, inspired by POD-DL-ROMs (and in general by the POD+DNN family of architectures) we rely on POD to compute the basis matrix $\mathbf{V} \in \mathbb{R}^{N_h \times N}$ (through SVD or its randomized version \cite{szlam2014}) which however only induces a {discrete lifting}. To obtain a set of continuous functions required for the residual computation in a strong continuous formulation, we train the trunk net $\hat{\mathbf{v}}(\mathbf{x}) \in \mathbb{R}^N$ to interpolate the POD modes, thus devising a supervised loss functional, namely,
\begin{equation}
\label{eq:loss_trunk}
    \mathcal{L}_{trunk} = \sum_{i \in \mathcal{I}_{train}} \|\hat{\mathbf{v}}(\mathbf{x}_i) - \mathbf{V}[i,:]\|^2,
\end{equation}
where $\mathcal{I}_{train} \subset \{1,\ldots,N_h\}$ and $\mathbf{V}[i,:]$ extracts the $i$-th row of $\mathbf{V}$. 

Then, we instantiate the branch net architecture, which induces a set of time-parameter-dependent coefficients, namely $(\muv,t) \mapsto \hat{\mathbf{q}}(\muv,t)$, as a DL-ROM\footnote{We refer to a DL-ROM architecture \cite{fresca2021dlrom} as to the one obtained by combining the autoencoder $\Psi' \circ \Psi$ and a deep feedforward neural network $\phi$ as in a POD-DL-ROM, however without including any preliminary dimensionality reduction.} since DL-ROMs {\em (i)} feature a lower complexity when compared to vanilla DNNs \cite{franco2023deep,brivio2023error}, and {\em (ii)} constrain the latent space to have a meaningful latent representation (of dimension $n \ll N$) of the available supervised data thanks to their autoencoder-based architecture. In particular, we define the encoder $\mathbb{R}^N \ni \mathbf{V}^T\mathbf{u}_h(\muv,t) \mapsto \Psi'(\mathbf{V}^T\mathbf{u}_h(\muv,t)) \in \mathbb{R}^n$, the reduced network as $\mathbb{R}^{p+1} \ni (\muv,t) \mapsto \phi(\muv,t) \in \mathbb{R}^n$ and the decoder as $\mathbb{R}^{n} \ni \phi(\muv,t) \mapsto \Psi(\phi(\muv,t)) \in \mathbb{R}^N$. Furthermore, we complement the DL-ROM based architecture with the trunk net, namely $u(\mathbf{x},\muv,t) \approx \hat{\mathbf{v}}(\mathbf{x}) \cdot \hat{\mathbf{q}}(\muv,t)$, thus obtaining a special instance of a DeepONet (see Fig.~\ref{fig:ptpi-dl-rom}).

\begin{figure}[ht!]
    \centering
    \includegraphics[width=0.75\textwidth]{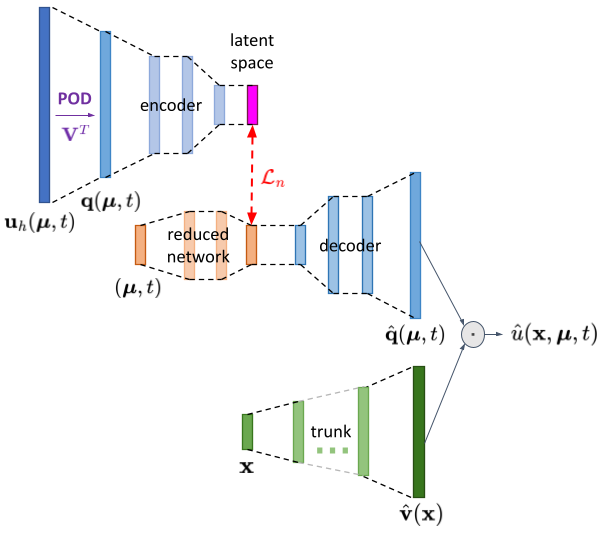}
    \caption{Schematic representation of the PTPI-DL-ROM architecture.}
    \label{fig:ptpi-dl-rom}
\end{figure}

\smallskip
\noindent \textbf{Step 2.} Before starting the subsequent phase of pre-training, we deem the trunk  net $\hat{\mathbf{v}}$ untrainable, namely, we ``freeze" its weights and biases, making them unaffected by any gradient descent procedure from now on. Thus, in the next stage, we only pre-train the branch net, \textit{de facto} decoupling the training of the trunk net and the branch net -- this is in contrast with the classic DeepONet training paradigm. Since the trunk network is ``frozen", we can pre-compute a set of quantities that will be useful when evaluating the loss function, involving both a data-driven and a physics-informed term. In particular, we store the matrix $\hat{\mathbf{V}}  \approx \mathbf{V} \in \mathbb{R}^{N_h \times N}$, computed as follows $[\hat{\mathbf{V}}]_{ik} = \hat{v}_k(\mathbf{x}_i)$ for $i=1,\ldots,N_h$ and $k=1,\ldots,N$. Then, we select a set of points $\{\mathbf{y}_i\}_{i=1}^{N_y} \subset \overline{\Omega}$ (that may coincide with the mesh vertices $\{\mathbf{x}_i\}_{i=1}^{N_h}$) to evaluate the physics-based residual at, and we compute offline the spatial derivatives entailed by the linear elliptic operator $\mathcal{A}$, the nonlinear operator $\mathcal{N}$ and the boundary operator $\mathcal{B}$.

We highlight that the obtained low-rank architecture is space-continuous, thus we can formulate a loss functional that blends data and physics as follows,
\begin{equation}
\label{eq:loss_ptpi}
    \mathcal{L}_{PTPI-DL-ROM} = \omega_N \mathcal{L}_N + \omega_n \mathcal{L}_n + \omega_{\Omega} \mathcal{L}_{\Omega} + \omega_{\partial \Omega} \mathcal{L}_{\partial \Omega}, 
\end{equation}
where
\begin{equation*}
\begin{aligned}
\mathcal{L}_N  &= \sum_{j=1}^{N_s}
         \sum_{k=1}^{N_t} 
         \|\hat{\mathbf{V}}\Psi \circ \phi(\muv_j,t_k) - \mathbf{u}_h(\muv_j,t_k)\|^2, \\
\mathcal{L}_{n} &=  
         \sum_{j=1}^{N_s} 
         \sum_{k=1}^{N_t} 
         \|\Psi'(\mathbf{V}^T\mathbf{u}_h(\muv_j,t_k)) - \phi(\muv_j,t_k)\|^2 \\
\mathcal{L}_{\Omega} &= \sum_{i=1}^{N_y} \sum_{j=1}^{N_{res}}
          |\mathcal{R}^{sc}_{\Omega}(\mathbf{y}_i,\muv_j,t_j)|^2 \mathsf{1}_{\mathbf{y}_i \in \Omega} \\
\mathcal{L}_{\partial \Omega} &= \sum_{i=1}^{N_y} \sum_{j=1}^{N_{res}}
         |\mathcal{R}^{sc}_{\partial \Omega}(\mathbf{y}_i,\muv_j,t_j)|^2\mathsf{1}_{\mathbf{y}_i \in \partial\Omega},
\end{aligned}
\end{equation*}
where $\mathcal{L}_N$ minimizes the reconstruction error, $\mathcal{L}_{n}$ ensures a meaningful latent representation, $\mathcal{L}_{\Omega}$ enforces the PDE residual, while $\mathcal{L}_{\partial \Omega}$ constrains the neural network to satisfy the boundary conditions.

The resulting pre-training is suitably informative, yet remarkably lightweight, because {\em (i)} we decoupled the training of the trunk network and the branch network, {\em (ii)} the branch pre-training performed in Step 2 informs the architecture with the underlying physics without computing the expensive spatial derivatives during the online training phase. However, to further enhance the accuracy of our predictions, we operate an additional fine-tuning of the architecture.

\subsection{The rationale behind the fine-tuning}

The prediction accuracy of the proposed architecture is strictly bounded by the representation capabilities of the global spatial modes entailed by the trunk network with respect to the test solution manifold. Since the trunk network is constrained to approximate the set of POD modes computed with the available supervised data, we further elaborate on both interpolation and extrapolation tasks in the time-parameter space from the POD perspective. In particular, we investigate the effect of the amount of available training data on the generalization capabilities of POD. To do that, we first suppose to have access to an infinite number of training data and define the optimal rank-$N$ (with $N \le N_h$) representation for the supervised training manifold $\mathcal{S}_{sup} = \{\mathbf{u}_h(\muv,t) \in \mathbb{R}^{N_h}: (\muv,t) \in \mathcal{P}_{sup} \times \mathcal{T}_{sup}\}$.

\begin{definition}
\label{def:optimality-pod}
    The basis $\mathbf{V}_{\infty} \in \mathbb{R}^{N_h \times N}$ is optimal with respect to the supervised training manifold $\mathcal{S}_{sup}$ if 
    \begin{equation*}
        \begin{aligned}
        \sum_{k>N} \sigma_{k,\infty}^2 &= \int_{\mathcal{P}_{sup} \times \mathcal{T}_{sup}} \|\mathbf{u}_h(\muv, t) - \mathbf{V}_{\infty}\mathbf{V}_{\infty}^T\mathbf{u}_h(\muv, t)\|^2  d(\muv,t) 
        \\ &= \min_{\bm{W} \in \mathbb{R}^{N_h \times N}: \mathbf{W}^T\mathbf{W} = \bm{I}} \int_{\mathcal{P}_{sup} \times \mathcal{T}_{sup}} \|\mathbf{u}_h(\muv, t) - \mathbf{W}\mathbf{W}^T\mathbf{u}_h(\muv, t)\|^2  d(\muv,t),
        \end{aligned}
    \end{equation*}
where
\begin{equation*}
    \mathbf{K}_{\infty} = \int_{\mathcal{P}_{sup} \times \mathcal{T}_{sup}} \mathbf{u}_h(\muv,t) \mathbf{u}_h(\muv,t)^T d(\muv,t) \in \mathbb{R}^{N_h \times N_h}
\end{equation*} 
is the correlation matrix and $\{\sigma_{k,\infty}^2\}_{k=1}^{N_h}$ are its eigenvalues.
\end{definition}

However, the optimal representation entailed by $\mathbf{V}_{\infty}$ is in general not available, requiring an infinite amount of data \cite{brivio2023error}: in practice, we have to rely on the POD basis computed through SVD, employing the so-called {\em method of snapshots} \cite{QMN_16}, and thus obtaining a quasi-optimal POD representation, according  to the following definition.

\begin{definition}
\label{def:quasi-optimality-pod}
    We deem the basis $\mathbf{V} := \mathbf{V}(\mathsf{P}_{sup} \times \mathsf{T}_{sup}) \in \mathbb{R}^{N_h \times N}$ quasi-optimal with respect to the supervised training manifold $\mathcal{S}_{sup}$ if
    \begin{equation*}
    \begin{aligned}
        \sum_{k>N} \sigma_k^2 &= \frac{|\mathcal{P}_{sup} \times \mathcal{T}_{sup}|}{N_sN_t} \sum_{(\muv,t) \in \mathsf{P}_{sup} \times \mathsf{T}_{sup}}\|\mathbf{u}_h(\muv,t) - \mathbf{V}\mathbf{V}^T\mathbf{u}_h(\muv,t)\|^2 \\ &= \min_{\mathbf{W} \in \mathbb{R}^{N_h \times N}: \mathbf{W}^T\mathbf{W} = \bm{I}} \frac{|\mathcal{P}_{sup} \times \mathcal{T}_{sup}|}{N_sN_t} \sum_{(\muv,t) \in \mathsf{P}_{sup} \times \mathsf{T}_{sup}}\|\mathbf{u}_h(\muv,t) - \mathbf{W}\mathbf{W}^T\mathbf{u}_h(\muv,t)\|^2,
        \end{aligned}
    \end{equation*}
where
\begin{equation*}
    \mathbf{K} = \frac{|\mathcal{P}_{sup} \times \mathcal{T}_{sup}|}{N_sN_t} \mathbf{u}_h\mathbf{u}_h^T \in \mathbb{R}^{N_h \times N_h}
\end{equation*} 
is the correlation matrix and $\{\sigma_k^2\}_{k=1}^{N_h}$ are its eigenvalues. 
\end{definition}

Despite POD provides the best possible representation,  among linear subspaces of a given dimension, of the training data at hand (due to the Eckart-Young-Schmidt theorem \cite{eckart1936}), the POD basis computed through SVD is only quasi-optimal. Indeed, at a $(\mathcal{P} \times \mathcal{T})$-continuous level, we have that
\begin{equation*}
    \int_{\mathcal{P}_{sup} \times \mathcal{T}_{sup}} \|\mathbf{u}_h(\muv, t) - \mathbf{V}_{\infty}\mathbf{V}_{\infty}^T\mathbf{u}_h(\muv, t)\|^2  d(\muv,t) \le \int_{\mathcal{P}_{sup} \times \mathcal{T}_{sup}} \|\mathbf{u}_h(\muv, t) - \mathbf{V}\mathbf{V}^T\mathbf{u}_h(\muv, t)\|^2  d(\muv,t) .
\end{equation*}

In this respect, we focus on the formulation of a convergence result that aims at bridging the gap between the optimal POD basis $\mathbf{V}_\infty$ and the  quasi-optimal POD basis $\mathbf{V}$. To do that, we define the generalization error of POD with respect the time-parameter space $\mathcal{P} \times \mathcal{T}$ as
\begin{equation*}
     \mathcal{E}_{POD}^{gen}(\mathcal{P} \times \mathcal{T};N_s,N_t) = \int_{\mathcal{P} \times \mathcal{T}} \|\mathbf{u}_h(\muv, t) - \mathbf{V}\mathbf{V}^T\mathbf{u}_h(\muv, t)\|^2 - \|\mathbf{u}_h(\muv, t) - \mathbf{V}_{\infty}\mathbf{V}_{\infty}^T\mathbf{u}_h(\muv, t)\|^2  d(\muv,t),
\end{equation*}
and we state the following result.
\begin{lemma}
\label{lemma:pod-convergence}
    For any $N \le \min\{N_h,N_sN_t\}$, denote by $\mathbf{V} \in \mathbb{R}^{N_h \times N}$ the POD basis computed through SVD, and by $\mathbf{V}_{\infty} \in \mathbb{R}^{N_h \times N}$ the optimal POD projection matrix. Then, there exists a sampling strategy for $(\muv,t) \in \mathcal{P}_{sup} \times \mathcal{T}_{sup}$ such that
    \begin{equation*}
         \mathcal{E}_{POD}^{gen}(\mathcal{P}_{sup} \times \mathcal{T}_{sup};N_s,N_t) \xrightarrow[]{a.s.} 0, \qquad N_s,N_t\rightarrow +\infty.
    \end{equation*}
\end{lemma}
\begin{proof}
    We report the proof in the Appendix.
\end{proof}

The latter result provides an insight on the performance of POD in the interpolation regime, where $\mathcal{P}_{test} \times \mathcal{T}_{test} \equiv \mathcal{P}_{sup} \times \mathcal{T}_{sup}$, so that $\mathcal{E}_{POD}^{gen}(\mathcal{P}_{test} \times \mathcal{T}_{test};N_s,N_t) = \mathcal{E}_{POD}^{gen}(\mathcal{P}_{sup} \times \mathcal{T}_{sup};N_s,N_t)$. Notably, from Lemma \ref{lemma:pod-convergence}, we can infer that in the interpolation regime we can expect that the more training data we have access to, the better the rank-$N$ representation entailed by POD in terms of retained variance. In other words, it is possible to make the generalization error on the test time-parameter space arbitrarily small. On the contrary, in the small data regime, where only a handful of scattered training time-parameters data points are available (namely, $|\mathsf{P}_{sup} \times \mathsf{T}_{sup}| = N_sN_t$ is small), POD might yield a suboptimal low-rank compression if the training dataset is not representative of the entire supervised training manifold $\mathcal{S}_{sup}$. 

Instead, when dealing with extrapolation regimes ($\mathcal{P}_{test} \times \mathcal{T}_{test} \supsetneq \mathcal{P}_{sup} \times \mathcal{T}_{sup}$) we have that 
\begin{equation*}
\begin{aligned}
&\mathcal{E}_{POD}^{gen}(\mathcal{P}_{test} \times \mathcal{T}_{test};N_s,N_t) = \mathcal{E}_{POD}^{gen}(\mathcal{P}_{sup} \times \mathcal{T}_{sup};N_s,N_t) + \\
& \hspace{5.5cm} + \mathcal{E}_{POD}^{gen}(\mathcal{P}_{test} \times \mathcal{T}_{test} \setminus \mathcal{P}_{sup} \times \mathcal{T}_{sup};N_s,N_t),
\end{aligned}
\end{equation*}
where we suppose that $|\mathcal{P}_{test} \times \mathcal{T}_{test} \setminus \mathcal{P}_{sup} \times \mathcal{T}_{sup}| > 0$.
We notice that, while we can sample suitably the supervised training dataset so that $\mathcal{E}_{POD}^{gen}(\mathcal{P}_{sup} \times \mathcal{T}_{sup};N_s,N_t)$ is arbitrarily small thanks to \ref{lemma:pod-convergence}, we cannot achieve the same conclusion for the remaining contribution of the generalization error referring to the test instances that belong to the extrapolation regime, namely $\mathcal{E}_{POD}^{gen}(\mathcal{P}_{test} \times \mathcal{T}_{test} \setminus \mathcal{P}_{sup} \times \mathcal{T}_{sup};N_s,N_t)$.

Then, once the test set becomes available, one may investigate \textit{a posteriori} the effect of suboptimality of the POD compression with a practical criterion, both in a small data interpolation and in an extrapolation regime. To this end, once the test instances become available, we can compute the empirical relative POD error as 
\begin{equation*}
    e_{POD}(\mathsf{D}_{test}) := \Biggl(\frac{\sum_{(\muv,t) \in \mathsf{D}_{test}}\|\mathbf{u}_h(\muv, t) - \mathbf{V} \mathbf{V}^T\mathbf{u}_h(\muv, t)\|^2}{\sum_{(\muv,t) \in \mathsf{D}_{test}}\|\mathbf{u}_h(\muv, t)\|^2}\Biggr)^{1/2}.
\end{equation*}

Lemma \ref{lemma:pod-convergence} and the subsequent remarks provide us with some insights on how small data interpolation and extrapolation regimes might impact on the representation capabilities of the POD modes and, by extension, of a trunk net approximating the POD modes. We emphasize that investigating the representation effectiveness of spatial global modes in the present context is of fundamental relevance since prediction accuracy of low-rank DL-based architectures is bounded by the low-rank projection error \cite{bhattacharya2021,brivio2023error}. For this reason, if the modes' representation capabilities are poor, the prediction accuracy might be unsatisfactory as well. However, by employing our pre-training algorithm, the trunk net modes are inferred only from the available data, and are not modulated by the physics. Indeed, the physics-based loss formulation plays a role only in Step 2 of the pre-training algorithm, where the trunk net is ``frozen", so that its weights and biases are not updated by gradient descent. 

In this respect, we equip our training procedure with a further fine-tuning that improves the representation capabilities of the global spatial modes and thus, by extension, enhances the accuracy and the reliability of the overall method. Specifically, we ``unfreeze" the trunk net, making its weights and biases susceptible to gradient descent updates. Then, we repeat Step 2 of pre-training, this time training simultaneously both the trunk and the branch networks. We stress that the fine-tuning is relatively more expensive since, differently from the pre-training phase, it entails the computation of the spatial derivatives during the training phase. On the other hand, the fine-tuning enables the trunk net and the branch net to be informed by physics at the same time. In this way, we can adjust simultaneously the modes entailed by the trunk net and the parametric map entailed by the branch net with the contribution of the physics-based residual corresponding to the samples drawn from the regions of the test time-parameter set not properly covered by the supervised training data, following the procedure explained in Section \ref{sec:mitigating-the-shortfalls-small-data}. The fine-tuning concludes the PTPI-DL-ROM training algorithm. For the complete training pseudocode, we refer the reader to Algorithm \ref{alg:pepi-dl-rom-training_alg}.

\begin{algorithm}[H]
    \label{alg:pepi-dl-rom-training}
    \caption{PTPI-DL-ROM training algorithm}
    \label{alg:pepi-dl-rom-training_alg}
    \begin{algorithmic}[1]
    \STATE \textbf{Inputs} Supervised training data $\mathsf{D}_{sup}$, unlabeled training data $\mathsf{D}_{res}$, set of mesh points $\{\mathbf{x}_i\}_{i=1}^{N_h}$, neural network architecture, reduced dimension $N$.
    \STATE Compute the POD basis $\mathbf{V} \in \mathbb{R}^{N_h \times N}$ with SVD using the supervised training data $\mathsf{D}_{sup}$.
    \STATE Pre-train the trunk net to approximate the POD basis with the loss functional \eqref{eq:loss_trunk}, using the labeled data $\{(\mathbf{x}_i,\mathbf{V}[i,:]\}_{i=1}^{N_h}$.
    \STATE ``Freeze" the trunk net's weights and biases.
    \STATE Pre-train the branch net with the loss functional \eqref{eq:loss_ptpi}, using $\mathsf{D}_{sup}$ as labeled data and $\mathsf{D}_{res}$ for the physics-based residual.
    \STATE ``Unfreeze" the trunk net's weights and biases.
    \STATE Fine-tune the branch net and the trunk net simultaneously with the loss functional \eqref{eq:loss_ptpi}, using $\mathsf{D}_{sup}$ as labeled data and $\mathsf{D}_{res}$ for the physics-based residual.
    \STATE \textbf{Output} Trained PTPI-DL-ROM architecture
    \end{algorithmic}
\end{algorithm}

\section{Numerical experiments}

In this section we assess the accuracy and the performance of the proposed PTPI-DL-ROM through a series of numerical experiments carried out on parametrized problems involving non-affine terms, non-stationary regimes, nonlinearities, or a combination of the former aspects. Specifically, we focus on:

\begin{itemize}
    \item[(i)] validating the proposed residual computation technique, assessing the reliability of the predictions on the test set;
    \item[(ii)] assessing the small data interpolation and extrapolation capabilities of the proposed PTPI-DL-ROMs;
    \item[(iii)] demonstrating the computational efficiency of the proposed training algorithm.
\end{itemize}

We employ generic dense layers, possibly with a Fourier feature expansion at the input \cite{wang2021learning, sharma2023stiff}, with sufficiently regular activations like ELU or SILU. Unless otherwise stated, we make use of the Adam optimizer to train neural networks, delegating $80\% \div 90\%$ of the supervised samples to the training set and the remaining samples to the validation set, through which we select the best model.

To measure the performance of the proposed ROM strategy on the test set, we choose to employ the following error indicators:
\begin{itemize}
\item the \textit{time-wise} relative error, which is employed to measure the accuracy of the predictions over time, namely,
\begin{equation*}
    E(t;N_{data}) := \frac{1}{N_s^{test}} \sum_{j=1}^{N_s^{test}} \biggl(\frac{\sum_{i=1}^{N_h} |u_{h,i}(\muv_j,t) - \hat{u}(\mathbf{x}_i,\muv_j,t)|^2}{\sum_{i=1}^{N_h} |u_{h,i}(\muv_j,t)|^2}\biggr)^{1/2};
\end{equation*}

\item the \textit{per-snapshot} relative error, which measures the variability of the error across the parameter space, that is
\begin{equation*}
    e(\muv;N_{data}) :=  \frac{(\sum_{k=1}^{N_t^{test}} \sum_{i=1}^{N_h} |\hat{u}_h(\mathbf{x}_i,\muv,t_k) - u_{h,i}(\muv,t_k)|^2)^{1/2}}{(\sum_{k=1}^{N_t^{test}} \sum_{i=1}^{N_h} |\mathbf{u}(\muv,t_k)|^2)^{1/2}};
\end{equation*}

\item the relative error over the entire time-parameter space, namely,
\begin{equation*}
    \mathcal{E}(N_{data}) := \frac{1}{N_s^{test}N_t^{test}} \sum_{j=1}^{N_s^{test}} \sum_{k=1}^{N_t^{test}} \frac{(\sum_{i=1}^{N_h} |u_{h,i}(\muv_j,t_k) - \hat{u}(\mathbf{x}_i,\muv_j,t_k)|^2)^{1/2}}{(\sum_{i=1}^{N_h} |u_{h,i}(\muv_j,t_k)|^2)^{1/2}}.
\end{equation*}
\end{itemize} 

We highlight the intrinsic dependence of all error indicators on the number of available supervised training samples $N_{data} = N_sN_t$.

\subsection{Eikonal Equation} 
The first test case deals with the Eikonal Equation, a stationary PDE with a nonpolynomial nonlinearity. Through the present test case we aim at demonstrating that, differently from data-driven POD-DL-ROMs, the proposed PTPI-DL-ROM is able to mitigate the shortfalls of small data regimes.

\smallskip
\noindent \textbf{Problem description.} We consider the problem
\begin{equation*}
     \left\{
    \begin{array}{rl}
    \displaystyle
    \|\nabla u\| = 1 & \quad \mbox{in } \Omega \\
    u = 0 & \quad \mbox{on } \Gamma(\mu),
    \end{array}
    \right.
\end{equation*}
where $\Omega = (-1,1)^2$ denotes the computational domain and $\Gamma := \Gamma(\mu) = \{\mathbf{x} = (x_1,x_2) \in \Omega \subset \mathbb{R}^2: x_1^2 + x_2^2 = \mu^2\}$ the circle of radius $\mu>0$. The exact solution of the problem above is given by
\begin{equation*}
    u_{exact}(x_1,x_2,\mu) = \mu - \sqrt{x_1^2+x_2^2}.
\end{equation*}

\smallskip
\noindent \textbf{Data collection and architectures.}
The problem does not depend on the time variable, thus we let $N_t = N_t^{test} = 1$. We evaluate the performance of our strategy using a test set that refers to the sampling $\mathsf{P}_{test} = \{0.13+0.05k, k =0,\ldots,17\} \subset \mathcal{P}_{test} = [0.1,1]$. To train the neural network, we collect a dataset of supervised samples referring to the parameter sampling $\mathsf{P}_{sup} = \{0.1+0.01k, k = 0,\ldots,40\} \subset \mathcal{P}_{sup} = [0.1,0.5]$. We remark that both the supervised data and the test set refer to the same mesh of $N_h = 900$ vertices on a 30-by-30 equispaced grid, while we choose $N = 2$ and $n = 2$ as POD and latent dimensions, respectively. Then, we design the trunk network, the encoder, the decoder, and the reduced network as dense ELU-neural networks consisting of $4$ hidden layers of $50$ neurons each.

\smallskip
\noindent\textbf{Training specifics.} 
We pre-train the trunk net for $3000$ epochs with a batch size of $10$ and a learning rate of $10^{-3}$, whereas the pre-training of the branch network involves $1000$ epochs, with a batch size of $1$ for the labeled data and $10$ for the physics-based loss term, employing a learning rate of $3 \times 10^{-4}$. We then decrease the learning rate to $1 \times 10^{-4}$ for the fine-tuning phase, which needs a total of $500$ epochs. We employ the loss functional \eqref{eq:loss_ptpi}, with the following choice of the hyperparameters that weight the four contributions: $\omega_N=\omega_n=\omega_{\Omega}=0.5$, $\omega_{\partial \Omega}=100$. 
Since $\mathcal{P}_{sup} \subsetneq \mathcal{P}_{test}$ we tackle an extrapolation task in the parameter space; thus, we uniformly sample the unlabeled dataset for the physics-based residual from $\mathcal{P}_{res} = [0.1, 1.1] \supset  \mathcal{P}_{test}$ every $5$ gradient descent epochs, collecting the realizations in $\mathsf{D}_{res} \subset \mathcal{P}_{res}$ with $ = |\mathsf{D}_{res}| = N_{res} = 1000$. 
We emphasize that we enforce the boundary conditions at $100$ uniformly spaced points on $\Gamma(\mu)$, whereas the collocation points for the imposition of the PDE residual are chosen as the mesh points during the pre-training phase and as $\{\mathbf{y}_i \sim \mathcal{U}(\Omega)\}_{i=1}^{1000}$ for the fine-tuning stage. 
We then compare the described PTPI-DL-ROM paradigm with a data-driven POD-DL-ROM architecture trained for $600$ epochs with a learning rate of $1 \times 10^{-3}$ and a batch size of $1$. We remark that the reported training specifics enable both PTPI-DL-ROM and POD-DL-ROMs to reach a suitable minimum in their loss functional.

\begin{figure}[b!]
    \centering
    \includegraphics[width=0.90\textwidth]{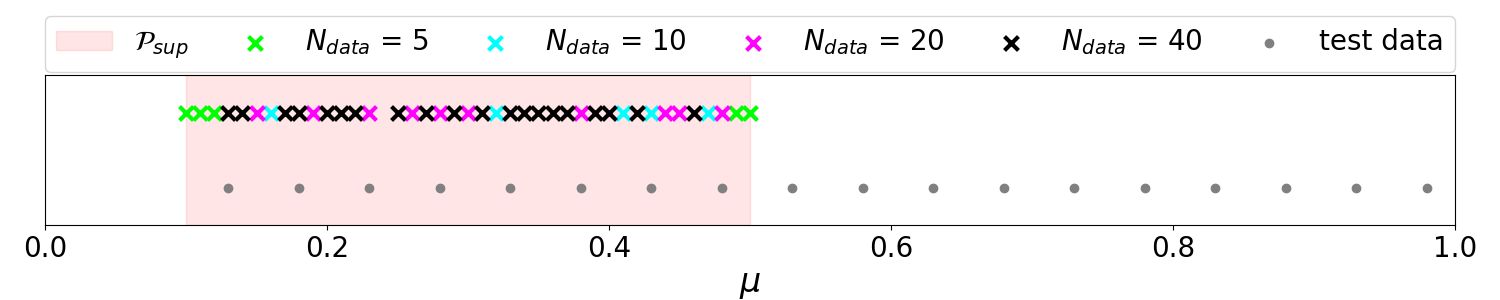}
    \caption{Eikonal equation: sampling strategy for the ablation study.}
    \label{fig:eikonal_visual}
\end{figure}

\smallskip
\noindent\textbf{Results analysis.} We measure how the proposed framework handles the small data regime by performing an ablation study. To do that, we iteratively construct  a group of supervised training datasets of increasing number of samples:
\begin{itemize}
    \item for $\mathsf{P}_{sup}^1$, we considered $N_{data}^1 = 5$, 
     selecting from $\mathsf{P}_{sup}$ the 5 outermost points.
    \item then, until depletion, we construct $\mathsf{P}_{sup}^i$ by adding $5 (2^{i-1} - 2^{i-2})$ snapshots to $\mathsf{P}_{sup}^{i-1}$, randomly sampling from $\mathsf{P}_{sup} \setminus \mathsf{P}_{sup}^{i-1}$.
\end{itemize}
Assembling the dataset in this way ensures consistency: for any $i \in \{1,2,3,4\}$ we have that $\min\{\mu \in \mathsf{P}_{sup}^i\} = \min\{\mu \in \mathcal{P}_{sup}\}$ and $\max\{\mu \in \mathsf{P}_{sup}^i\} = \max\{\mu \in \mathcal{P}_{sup}\}$. We display in Fig.~\ref{fig:eikonal_sampling} a realization of the proposed sampling technique.  
\begin{figure}[ht!]
    \centering
    \includegraphics[width=0.99\textwidth]{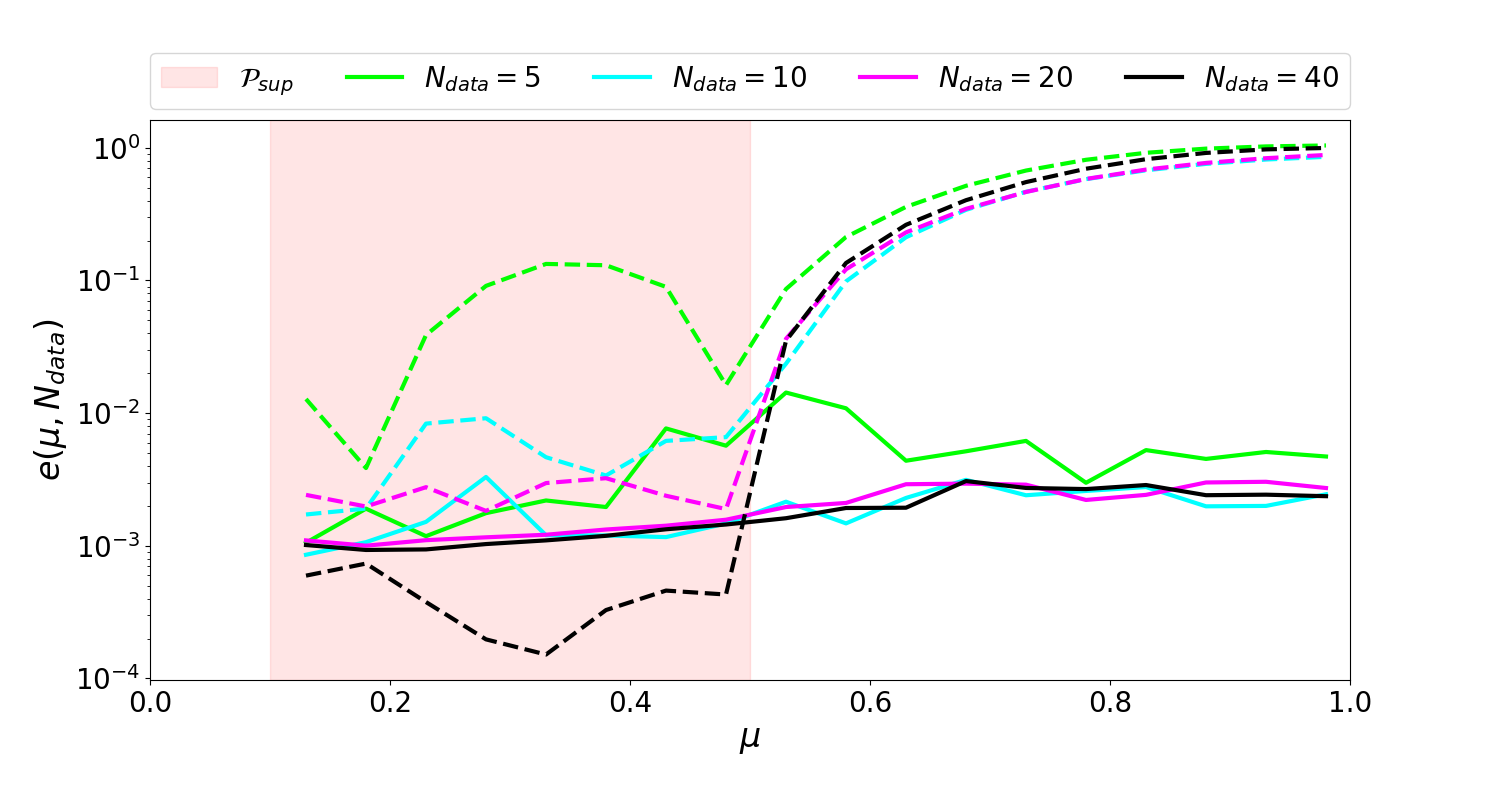}
    \caption{Eikonal equation: results of the ablation study. The continuous line refers to the novel PTPI-DL-ROM paradigm, whereas the dashed line concerns the data-driven POD-DL-ROM.}
    \label{fig:eikonal_sampling}
\end{figure}
Then, for any dataset $\mathsf{P}_{sup}^i$, we train both the PTPI-DL-ROM and the POD-DL-ROM with the training specifics characterized beforehand. Moreover, we compare the accuracy of the two architectures measuring the \textit{per-snapshot} relative error over the test set, displaying the results in Fig.~\ref{fig:eikonal_acc}. In particular, we note that:
\begin{itemize}
    \item[(i)] PTPI-DL-ROMs compensate for a shortage of supervised data in the interpolation regime. Indeed, we when performing an interpolation task, as the number of supervised data decreases, the discrepancy between the PTPI-DL-ROMs and the data-driven POD-DL-ROMs, in terms of test set accuracy, increases;
    \item[(ii)] as the amount of supervised data increases, the PTPI-DL-ROM approach and the data-driven POD-DL-ROM reach extremely similar results when dealing with interpolation tasks. In particular, this validates the residual computation since it demonstrates a compatibility between the supervised data and the calculation of the residual, thus suggesting that they convey the same information;
    \item[(iii)] in the extrapolation regime we notice that error diverges in the case of the data-driven paradigm. As expected, moving further from the supervised domain yields worse accuracies. On the other hand, we emphasize PTPI-DL-ROM extrapolation capabilities: if suitably trained, the PTPI-DL-ROM approach achieves an almost constant accuracy on the entire test set. 
\end{itemize}

\begin{figure}[ht!]
    \centering
    \includegraphics[width=0.99\textwidth]{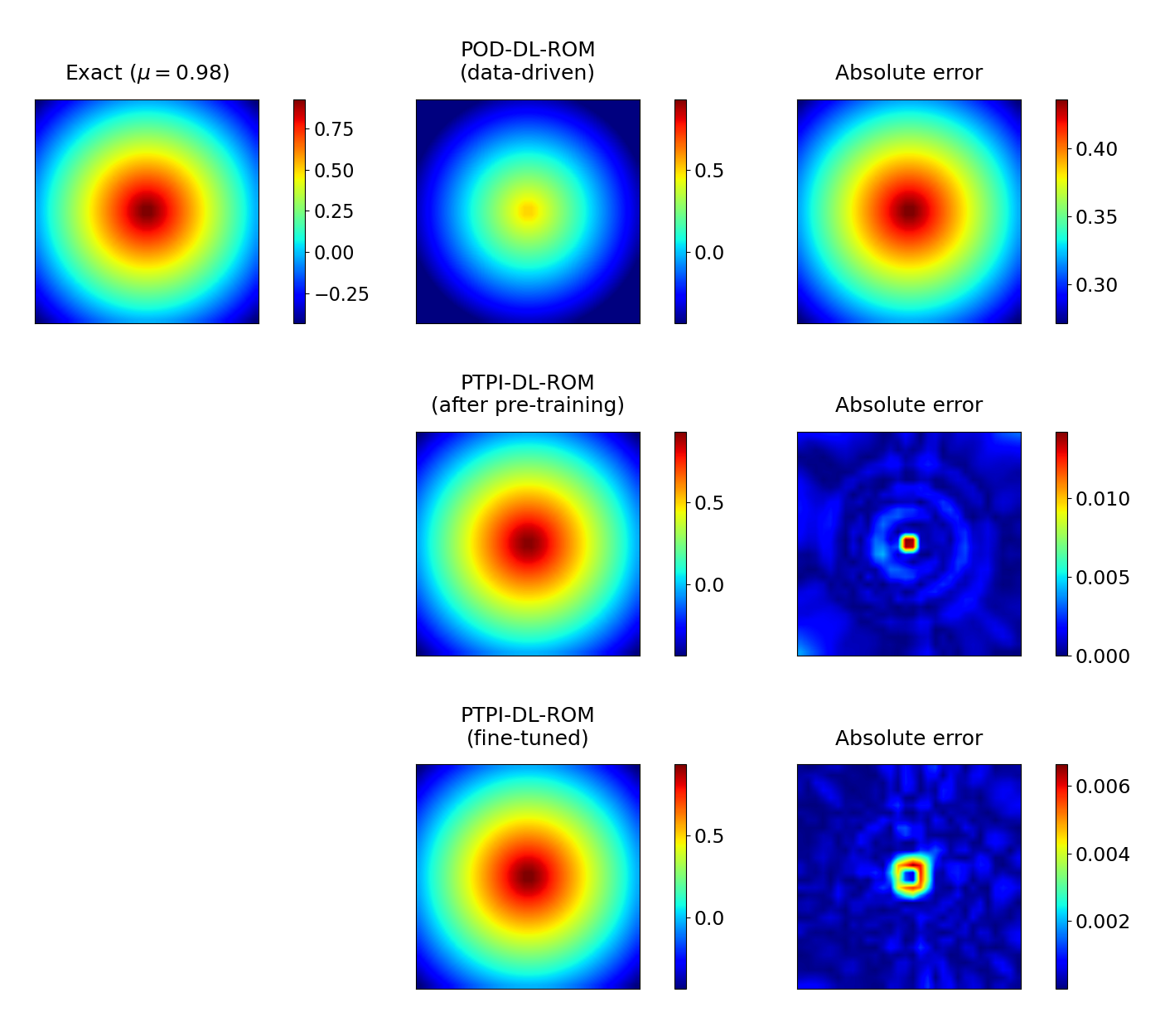}
    \caption{Eikonal equation: comparison between data-driven POD-DL-ROMs and the novel PTPI-DL-ROM for an instance that does not belong to the supervised training parameter space, namely, $\mu = 0.98 \in \mathcal{P}_{test} \setminus \mathcal{P}_{sup}$.}
    \label{fig:eikonal_acc}
\end{figure}
The remarkable extrapolation capabilities of the PTPI-DL-ROM paradigm are demonstrated also in Fig.~\ref{fig:eikonal_acc}. Specifically, we select the parameter instance as follows, $\mu = 0.98 \in \mathcal{P}_{test} \setminus \mathcal{P}_{sup}$. Then, we compare the prediction of data-driven POD-DL-ROMs with two results obtained with the novel PTPI-DL-ROM architecture (after the pre-training and after the fine-tuning procedure, respectively). We note that a total absence of supervised data in a neighborhood of $\mu = 0.98$ prevents the data-driven POD-DL-ROM to provide reliable results. On the other hand, the PTPI-DL-ROM predictions after pre-training are already adequately accurate since $e_{POD} \le 1 \times 10^{-6}$, namely, the POD basis is representative also of the extrapolation regime. Nonetheless, the simultaneous training of the branch net and the trunk net performed during the fine-tuning is able to further enhance the accuracy of a factor higher than 2 -- from $\mathcal{E}(N_{data} = 40) = 3.6 \times 10^{-3}$ after pre-training to $\mathcal{E}(N_{data} = 40) = 1.6 \times 10^{-3}$ after fine-tuning.

\subsection{Advection Diffusion Reaction Equation}

This second test case deals with a linear time-dependent advection-diffusion-reaction equation set on a three-dimensional domain, that depends nonaffinely on a vector of input parameters. Here we focus on showcasing the capability of the novel architecture to handle complex parameter-extrapolation tasks for 3D problems with high parametric variability, as well as  %Finally, we empirically demonstrate
the computational efficiency of the pre-training procedure of PTPI-DL-ROMs.% by a suitable comparison with baselines.

\smallskip
\noindent\textbf{Problem description.} The problem formulation reads as follows:
\begin{equation}
     \left\{
    \begin{array}{rll}
    \displaystyle
    \frac{\partial{u}}{\partial{t}} - 0.05 \Delta u + 0.05u + a(t)\frac{\partial{u}}{\partial{x}} &= f(x,y,z;\bm{\mu}), & \mbox{in } \Omega\times[0,T] \smallskip \\
    u &= 0, & \mbox{on }\partial{\Omega} \smallskip  \\
    u(t=0) &= 0, & \mbox{in } \Omega,
    \end{array}
    \right.
\end{equation}
where $\Omega = (0,1)^3$ is the spatial domain, $T = 7 \pi$, while the unknown $u = u(x,y,z, \bm{\mu}, t)$ -- which might represent the concentration of a chemical species -- is subject to the action of the time-dependent advection field $a(t) = \log(0.1t)$. Moreover, we choose as source term the $\bm{\mu}$-dependent function $f(x,y,z;\bm{\mu}) = e^{3xyz} \sin(\pi \mu_1 x) \sin(\pi \mu_2 y)$, showing a non-affine dependence on $\bm{\mu}=(\mu_1,\mu_2)$. 

\newpage
\noindent\textbf{Data collection and architectures.}
In order to collect the FOM snapshots, we spatially discretize the problem using a mesh made by $N_h=14075$ vertices and  formulate the semi-discrete high-fidelity problem by means of $\mathbb{P}^1$-FEM, using the \texttt{redbkit} library \cite{negri2016redbkit}. We opt for a BDF2 time advancing scheme with a constant time step $\Delta t = \pi / 10$ for a total of $N_t = 70$ time steps. Both the train and test samples refer to the same time set, namely $\mathsf{T}_{sup} = \mathsf{T}_{test} = \{\pi / 10 +  \pi / 10k, k=0,\ldots,N_t\}$. The parameter training instances are represented by the set $\mathsf{P}_{sup} = \{0.5 + 0.125k, k = 0, \ldots, 4\}^2 \subset [0.5,1]^2$, while the test parameter instances are given by $\mathsf{P}_{test} = \{0.55 + 29/140 k, k = 0, \ldots, 7\}^2 \subset [0.5,2]^2$. We define the supervised time-parameter space as $\mathcal{P}_{sup} \times [0,T] \supset \mathsf{P}_{sup} \times \mathsf{T}_{sup}$, and the test time-parameter space as $\mathcal{P}_{test} \times [0,T] \supset \mathsf{P}_{test} \times \mathsf{T}_{test}$.
We mention that, aiming to capture the variability of the solution manifold with sufficient accuracy, we choose $N = 15$ and $n = 8 \ge p + 1 = 3$. Then, we design our architecture as follows: the trunk network consists of $6$ hidden layers of $100$ units each with a sinusoidal nonlinearity; the branch encoder, decoder and reduced network are endowed with $4$ hidden layers of $50$ neurons each; the ELU activation function is used.

\smallskip
\noindent\textbf{Training specifics.}
The PTPI-DL-ROM training consists of a pre-training phase for the trunk network of $1000$ epochs, with a batch size of $100$ and a learning rate of $1 \times 10^{-3}$. Then, we pre-train the branch net with $2000$ gradient descent epochs employing $1 \times 10^{-3}$ as learning rate. Finally, we fine-tune the whole network with $4000$ epochs, by lowering the learning rate to $1 \times 10^{-4}$.
Since the present numerical experiment tackles an extrapolation task, we select $\mathcal{P}_{res} \times \mathcal{T}_{res} = [0.4,2.1]^2 \times [0,7.1\pi]$. 
In the loss function, the four contributions are weighted by $\omega_{N} = \omega_{n} = \omega_{\Omega} = 0.5$, $\omega_{\partial \Omega} = 50$, setting the resampling frequency of the unlabeled dataset $\mathsf{D}_{res} = \{(\muv,t)_j \in \mathcal{U}(\mathcal{P}_{res} \times \mathcal{T}_{res})\}_{j=1}^{100}$ to $5$ epochs.
We stress that in the pre-training phase we choose the mesh points as spatial collocation points for the residual computation; on the other hand, at the fine-tuning stage the PDE residual is enforced at $300$ uniformly sampled spatial points in $\Omega$, whereas the boundary conditions are still enforced at the boundary mesh points. We emphasize that the number of spatial collocation points needed for the evaluation of the residual is much lower than the amount of dofs of the high-fidelity labeled data.
Moreover, unless otherwise stated, the batch size is set to $30$ for the supervised data and to $10$ for the unlabeled data. 

Aiming at assessing the effectiveness of our proposed pre-training algorithm, we compare the PTPI-DL-ROM algorithm with two different training strategies. Specifically, we employ the same architecture and configuration as before, for the sake of fairness, considering

\begin{itemize}
    \item The \textit{w/o pre-training} approach, which entails a basic optimization procedure that does not involve any kind of pre-training. Indeed, this strategy implies only the minimization with respect to the physics-informed loss functional described in \eqref{eq:loss_ptpi}, employing $10000$ epochs of Adam and $1 \times 10^{-4}$ as learning rate. We remark that this approach was originally conceived in the context of DeepONets \cite{wang2021learning}.
    \item  The \textit{vanilla pre-training} approach \cite{zhu2023reliable}, which entails { \em (i) } a pre-training phase that simultaneously optimizes both the branch net and the trunk net with respect to a data-driven loss formulation (see Eq. \eqref{eq:loss_deeponet_supervised}), employing only the available labeled data, thus without any physics-based constraint, and { \em (ii)} a fine-tuning stage based on the physics-informed loss functional \eqref{eq:loss_ptpi}. We remark that the pre-training consists of a total of $3000$ epochs with a learning rate of $1 \times 10^{-3}$ and that we fine-tune the architecture for $7000$ epochs, using a lower learning rate ($1 \times 10^{-4}$).
\end{itemize}
We summarize the main features of the \textit{w/o pre-training}, the \textit{vanilla pre-training}, and the PTPI-DL-ROM optimization strategies in Table \ref{tab:comparison-pre-training}.

\begin{table}[ht!]
    \centering
    \begin{tabular}{|c|c|c|c|}
    \hline
         & \textit{w/o pre-training} & \textit{vanilla pre-training} & PTPI-DL-ROM\\
         \hline
         pre-training & - & data & data+physics \\
         fine-tuning & data+physics & data+physics & data+physics\\
         \hline
    \end{tabular}
    \caption{Comparison between the \textit{w/o pre-training} approach, the \textit{vanilla pre-training} strategy and the PTPI-DL-ROM paradigm.}
    \label{tab:comparison-pre-training}
\end{table}

Finally, for the sake of the comparison of the novel PTPI-DL-ROM against the data-driven POD-DL-ROM, we train this latter with $2000$ epochs of gradient descent and a learning rate of $10^{-3}$.

\begin{figure}[htb!]
    \centering
    \includegraphics[width=0.99\textwidth]{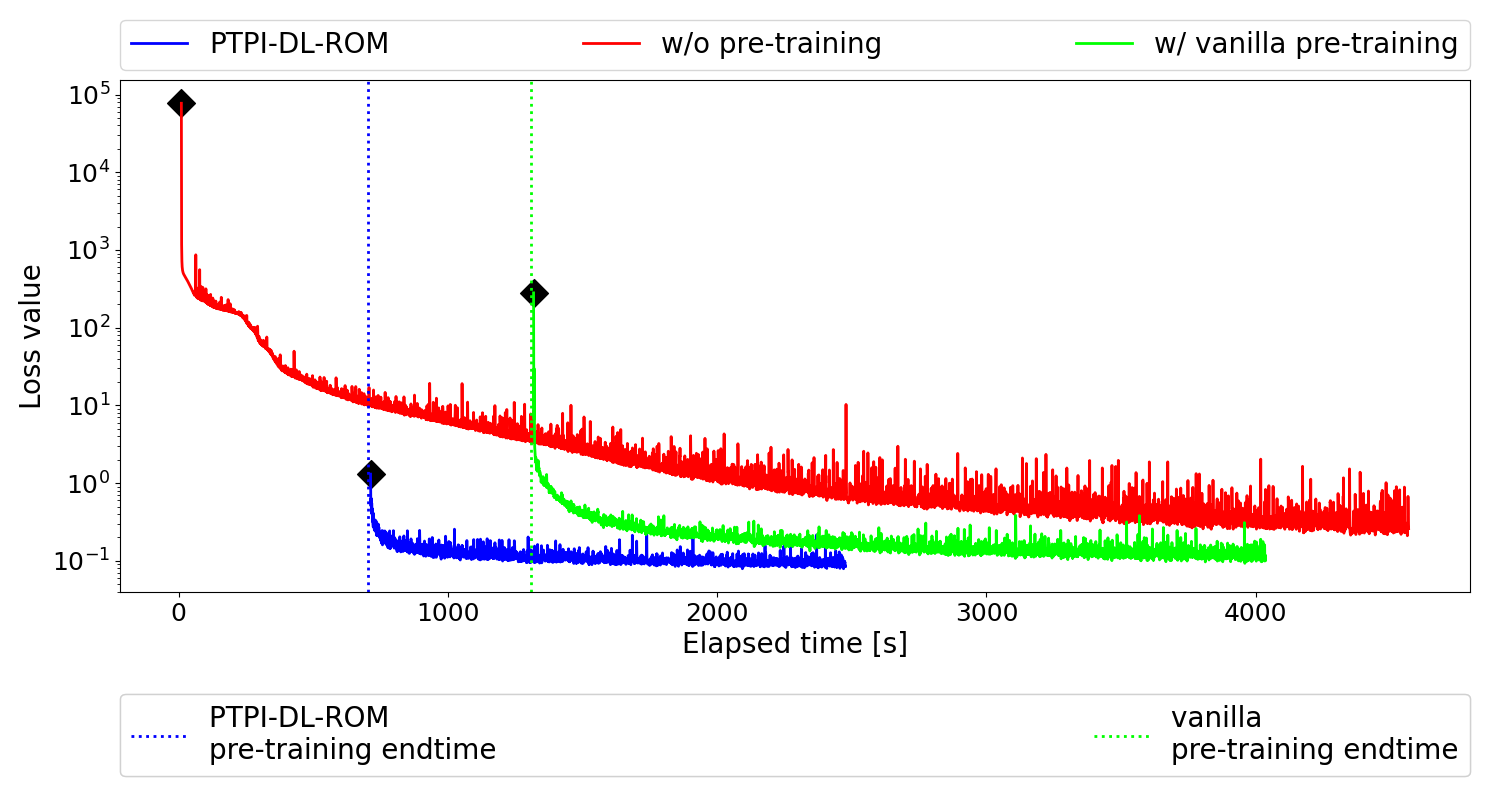}
    \caption{Advection Diffusion Reaction equation: comparison between different pre-training strategies in terms of their effect on the fine-tuning training loss decay. The relative error metric for the PTPI-DL-ROM paradigm is $\mathcal{E}(N_{data}) = 1.3 \times 10^{-1}$, whereas both the non-pre-trained approach and the vanilla approach result in worse accuracy ($\mathcal{E}(N_{data}) = 2.2 \times 10^{-1}$ and $\mathcal{E}(N_{data}) = 1.6 \times 10^{-1}$, respectively).}
    \label{fig:adr_training}
\end{figure}

\smallskip
\noindent\textbf{Results analysis.} The strong influence of the (nonaffinely parametrized) source term on the problem solution, as well as the presence of a time-dependent advection field, makes this problem a good test bed to compare the novel PTPI-DL-ROM paradigm and its data-driven counterpart. Indeed, for this test case, the \textit{per-snapshot} error indicator $e(\mu,N_{data})$ shows a high sensitivity depending on whether the spatial frequency effects are suitably captured. Indeed, as shown in Fig.~\ref{fig:adr_visual} and Fig.~\ref{fig:adr_err}, data-driven POD-DL-ROMs fail to achieve a reliable solution for sample instances that fall beyond the region from where labeled data are available. On the other hand, the proposed PTPI-DL-ROM architecture is able to return remarkably faithful reproductions of the test samples belonging to both the interpolation and the extrapolation regime. Specifically, Fig.~\ref{fig:adr_err} shows that in the interpolation regime the fine-tuned PTPI-DL-ROM integrates the knowledge available from the labeled data with the information about the underlying physics to reach an outstanding precision. Nonetheless, the PTPI-DL-ROM mitigates the absence of data with physics in the regions where labeled data are not available, thus providing reliable solutions in the extrapolation regime,  too. 

Ultimately, we can investigate the effectiveness of our proposed pre-training strategy: from Fig.~\ref{fig:adr_training} we notice that, in general, employing a pre-training strategy is essential to speed up the convergence to suitable minima. Nevertheless, the proposed PTPI-DL-ROM pre-training -- among the options we considered -- offers the most convenient strategy. Indeed, the decoupled nature of the PTPI-DL-ROM pre-training ensures lower computational costs than the vanilla pre-training strategy; conversely, this latter entails the simultaneous training of branch and trunk. More remarkably, we emphasize that the loss value computed at the first epoch of the fine-tuning is lower in the PTPI-DL-ROM case: indeed, the vanilla pre-training strategy is purely data-driven and thus more sensible to small data regimes, whereas the PTPI-DL-ROM pre-training is based upon a loss functional that combines data and physics. As a final remark, the PTPI-DL-ROM loss decay is the least prone to oscillations and thus most stable among the three proposed approaches.

\begin{figure}[htb!]
    \centering
    \includegraphics[width=0.99\textwidth]{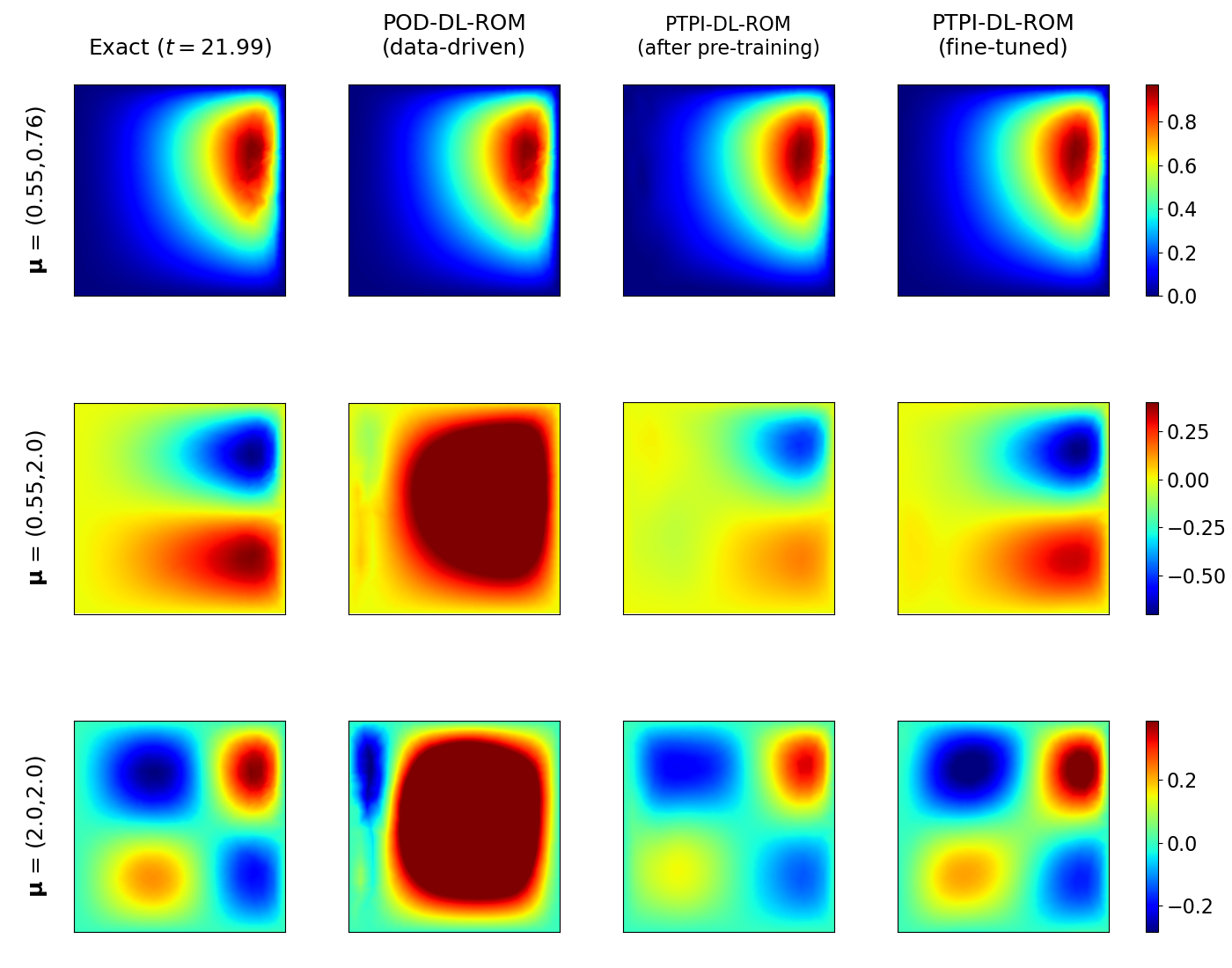}
    \caption{Advection Diffusion Reaction equation: comparison between the predictions obtained with data-driven POD-DL-ROM and the proposed PTPI-DL-ROM. The last two rows refers to instances pertaining to the extrapolation regime.}
    \label{fig:adr_visual}
\end{figure}

\begin{figure}[htb!]
    \centering
    \includegraphics[width=1.05\textwidth]{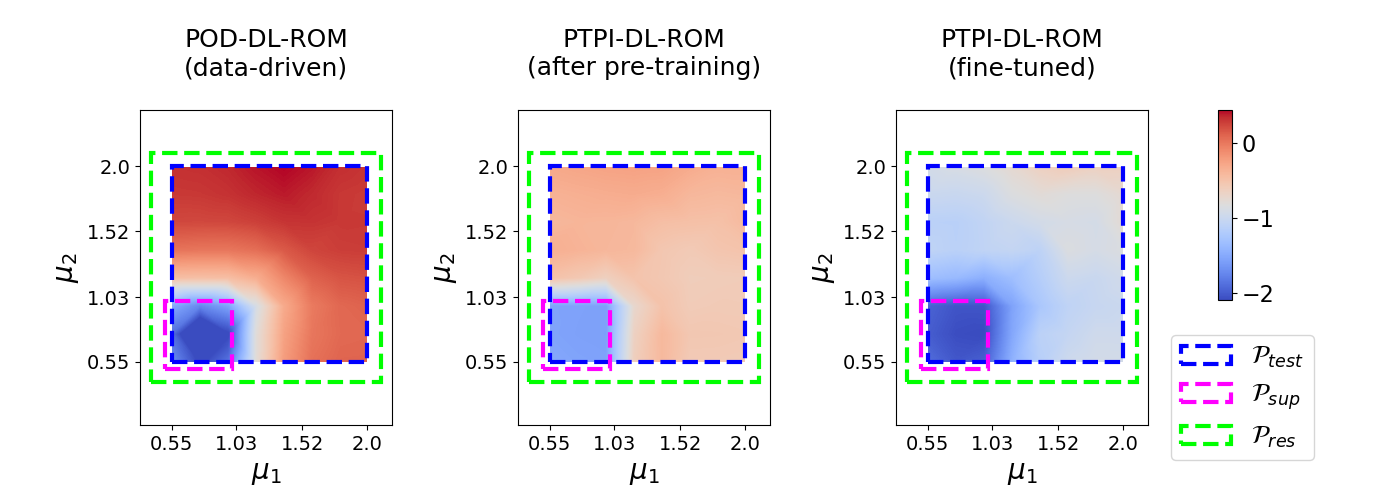}
    \caption{Advection Diffusion Reaction equation: visualization of $\log_{10}(e(\muv,N_{data}))$ for the data-driven POD-DL-ROM and the proposed PTPI-DL-ROM.}
    \label{fig:adr_err}
\end{figure}

\subsection{Darcy Equations}

This third test case aims at showing the capability of the proposed framework to both handle non-trivial vector problems and tackle complex parameter extrapolation tasks. 

\smallskip
\noindent\textbf{Problem description.} We consider the Darcy equations to describe a flow in a porous medium. In mixed formulation, the problem reads as follows:
\begin{equation}
\label{eq:darcy-problem}
    \left\{
    \begin{array}{rll}
    \displaystyle
    \kappa(\muv)^{-1} \bm{\sigma} - \nabla p &= 0 ,& \mbox{in } \Omega\\
    \nabla \cdot \bm{\sigma} + 1 &= 0 ,& \mbox{in } \Omega\\
    p &= 0,& \mbox{in } \Gamma_p\\
    \bm{\sigma} \cdot \mathbf{n} &= 0,& \mbox{in } \Gamma_\sigma,
    \end{array}
    \right.
\end{equation}
where the computational domain is $\Omega = (0,3) \times (0,1)$, whereas the domain boundary is divided into $\Gamma_p = [0,3] \times \{0,1\}$ and $\Gamma_\sigma = \partial \Omega \setminus \Gamma_p$.
We highlight that the problem is parameterized by means of the nonaffine permeability function $\kappa = \kappa(x,y;\muv) = \exp(\sin(\mu_1 \pi x) + \sin(\mu_2 \pi y))$ and that $\mathbf{n}$ denotes the outward normal unit vector.

\smallskip
\noindent\textbf{Data collection and architectures.}
We collect a total of $N_{data} = 196$ supervised snapshots corresponding to the following choice of the parameter train dataset $\mathsf{P}_{sup} = \{(1.0 + j/26, 1.0 + k/13), \mbox{ for }j,k=0,\ldots,13\} \subseteq [1.0,1.5] \times [1.0,2.0] = \mathcal{P}_{sup}$. To generate the supervised data, we discretize the problem \eqref{eq:darcy-problem} by means of the Finite Element method. In particular, to ensure inf-sup stability, we choose to employ Brezzi-Douglas-Marini elements of degree $1$ for the flux $\bm{\sigma}$, whereas we select discontinuous elements of degree $0$ for the pressure $p$. We remark that the resulting discretized system features $31756$ dofs, but, in order to avoid retaining redundant information, we choose to evaluate the solution fields only at the mesh vertices, thus yielding $N_h=12213$. 
We generate the test dataset by means of the same solver and configuration as the training dataset, by sampling the test parameter instances from $\mathsf{P}_{test} = \{(1.0 + j/26, 1.0 + k/13), \mbox{ for } j,k=0,\ldots,13\} \subset [1.0,1.5] \times [1.0,2.0] = \mathcal{P}_{test}$. We highlight the complexity of the parameter extrapolation task: we have that $\mathcal{P}_{sup} \subsetneq \mathcal{P}_{test}$ and $|\mathcal{P}_{test}| = 4 \times |\mathcal{P}_{sup}|$.
We select $N=30$ POD modes per channel, whereas the latent dimension is set to $n=7$, thus enabling suitable linear and nonlinear compressions of the solution manifold, respectively. 

Then, we construct the neural network core as follows: {\em (i) } the trunk network is a Fourier-features-enhanced dense network of $6$ hidden layers consisting of $100$ neurons each with a SILU nonlinearity, {\em (ii) } the branch encoder is a ELU-dense network of $4$ hidden layers of $50$ neurons each, {\em (iii) } the branch reduced network consists of $6$ hidden layers of $50$ neurons each with a ELU activation, and {\em (iv) } the branch decoder is composed of $6$ hidden layers of $70$ units each and is endowed with a ELU activation.

\smallskip

\noindent\textbf{Training specifics.}  Since the present numerical experiment entails an extrapolation task ($\mathcal{P}_{sup} \subsetneq \mathcal{P}_{test}$), we suitably sample the unlabeled dataset for the physics-based soft constraint as follows $\mathsf{D}_{res} = \{\muv_j \in \mathcal{U}(\mathcal{P}_{res})\}_{j=1}^{N_{res}}$, where $\mathcal{P}_{res} = [0.9,2.1] \times [0.9,3.1] \supset \mathcal{P}_{test}$. For the pre-training phase, we impose the residual at the mesh points, while at the fine-tuning stage we choose to sample $300$ collocation points for the imposition of the PDE residual from a uniform distribution over $\Omega$, and we enforce the boundary conditions at the boundary mesh points. We weight the loss contributions of the four terms of the loss formulation \eqref{eq:loss_ptpi} with the choice $\omega_N=\omega_n=\omega_{\Omega}=0.5$, $\omega_{\partial \Omega}=50$. 
The trunk net pre-training consists of 600 epochs, with a learning rate of $1 \times 10^{-3}$ and a batch size of 10. Then, we pre-train the branch net for $2000$ epochs, with a learning rate of $1 \times 10^{-3}$, and a batch size of $10$ for both the supervised and the unsupervised datasets, resampling $\mathsf{D}_{res}$ each $50$ epochs. Finally, we fine-tune the architecture for $1500$ epochs, decreasing the learning rate to $3 \times 10^{-4}$ and diminishing the resampling frequency for $\mathsf{D}_{res}$ to $5$. For the sake of comparison, we stress that the data-driven POD-DL-ROM is trained for $600$ epochs, employing a learning rate of $1 \times 10^{-4}$ and a batch size of $10$.

\noindent\textbf{Results analysis.} According to the \textit{per-snapshot} relative errors reported in Fig.~\ref{fig:darcy_err}, we can remark that:
\begin{itemize}
    \item The data driven POD-DL-ROM performs well in the interpolation regime; on the other hand, moving further beyond the interpolation regime we obtain more and more skewed predictions.
    \item The pre-training phase of PTPI-DL-ROM mitigates the absence of data in the extrapolation regime thanks to its physics-informed loss formulation, however the error remains still high beyond the interpolation regime. This is due to the fact that $e_{POD} = 1.5 \times 10^{-1}$ is rather large: the POD modes are not adequately representative of the entire test parameter space.
    \item The fine-tuned PTPI-DL-ROM further refines the pre-trained architecture, thus enhancing the faithfulness to the high-fidelity solution: indeed, the simultaneous physics-informed training of the branch and the trunk networks enables an adequate low-rank representation also beyond the interpolation regime. We emphasize that the fine-tuned PTPI-DL-ROM outperforms the data-driven POD-DL-ROM in terms of interpolation accuracy, since it integrates data and physics in the interpolation regime. Nonetheless,  PTPI-DL-ROMs are reliable even in the regions of the parameter space where data is completely missing, in contrast to data-driven POD-DL-ROMs.
\end{itemize}

\begin{figure}[h!tb]
    \centering
    \includegraphics[width=1.05\textwidth]{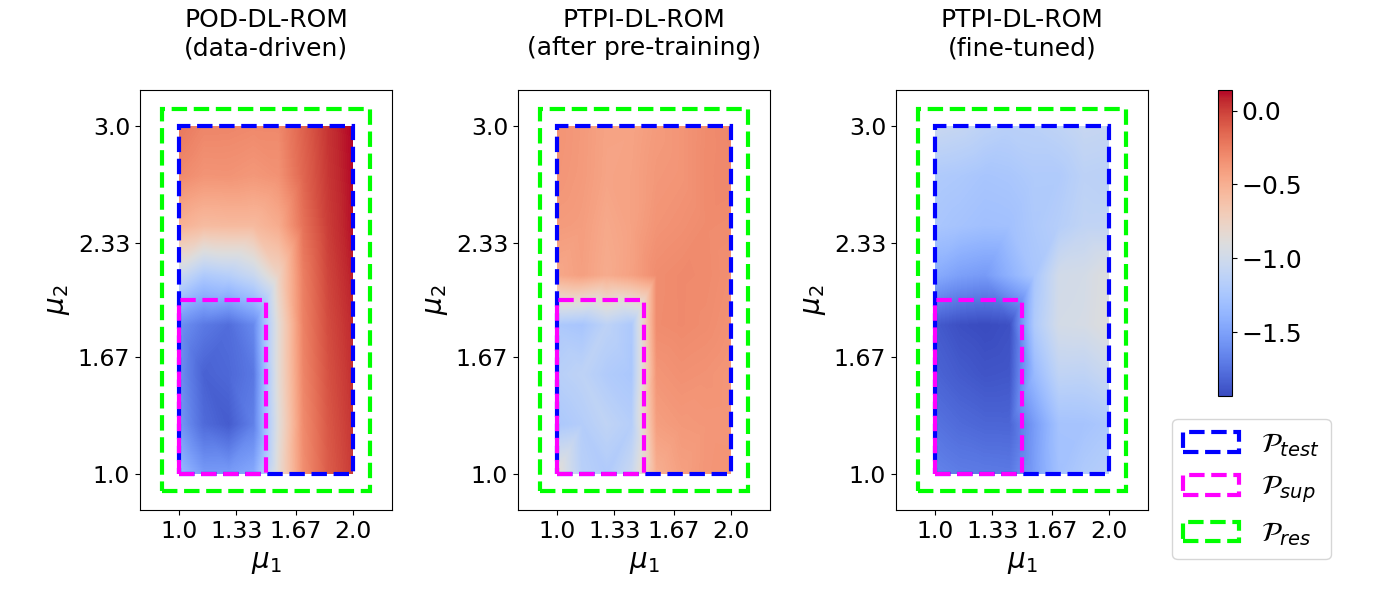}
    \caption{Darcy equations: visualization of $\log_{10}(e(\muv,N_{data}))$ for the data-driven POD-DL-ROM and the novel PTPI-DL-ROM.}
    \label{fig:darcy_err}
\end{figure}

\begin{figure}[h!tb]
    \centering
    \includegraphics[width=0.975\textwidth]{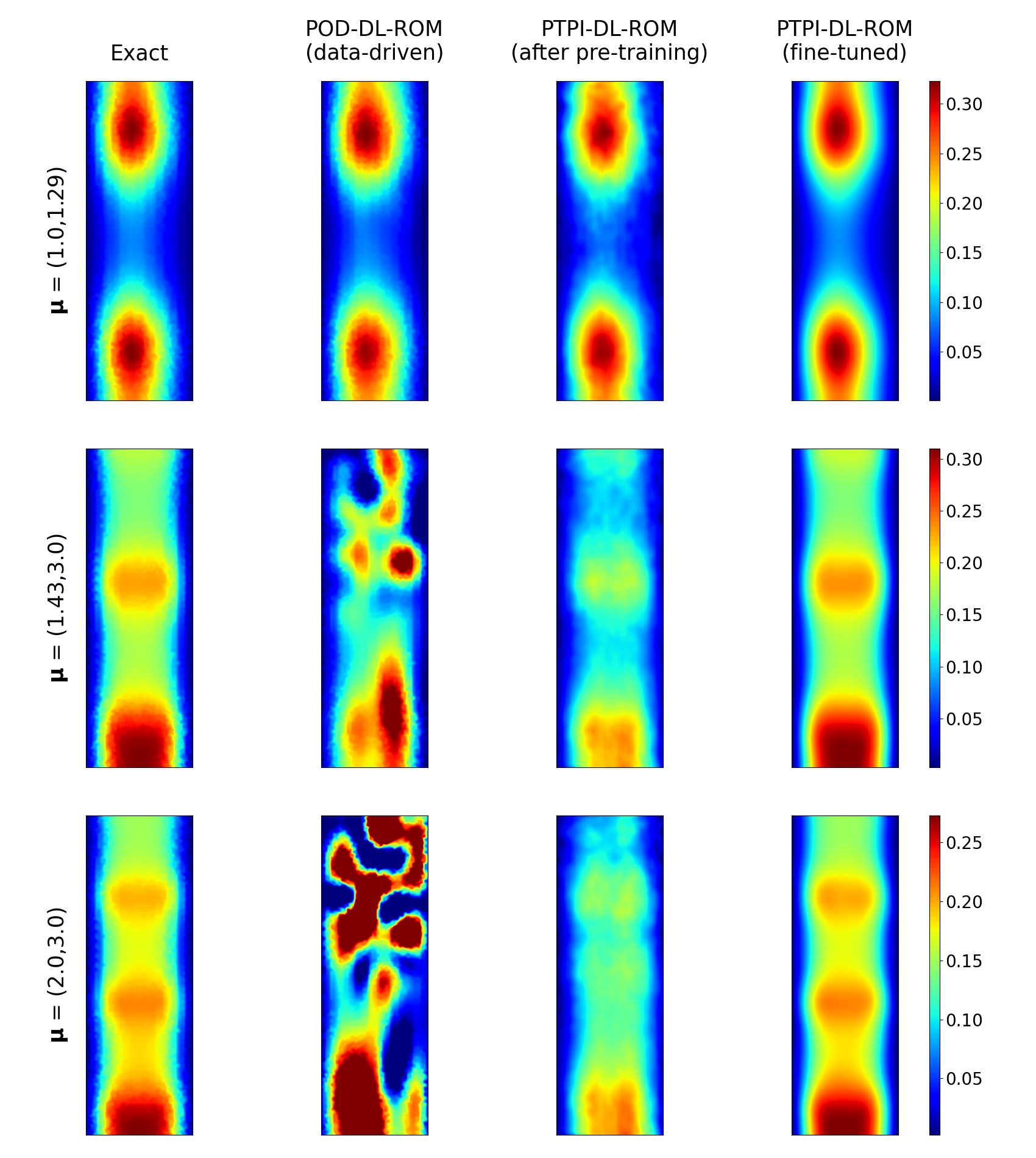}
    \caption{Darcy equations: visualization of the pressure field $p$ in the $yx$ plane for three different instances of the parameter $\muv$. We remark that the last two instances belong to the extrapolation regime.}
    \label{fig:darcy_pres}
\end{figure}

\begin{figure}[h!tb]
    \centering
    \includegraphics[width=0.975\textwidth]{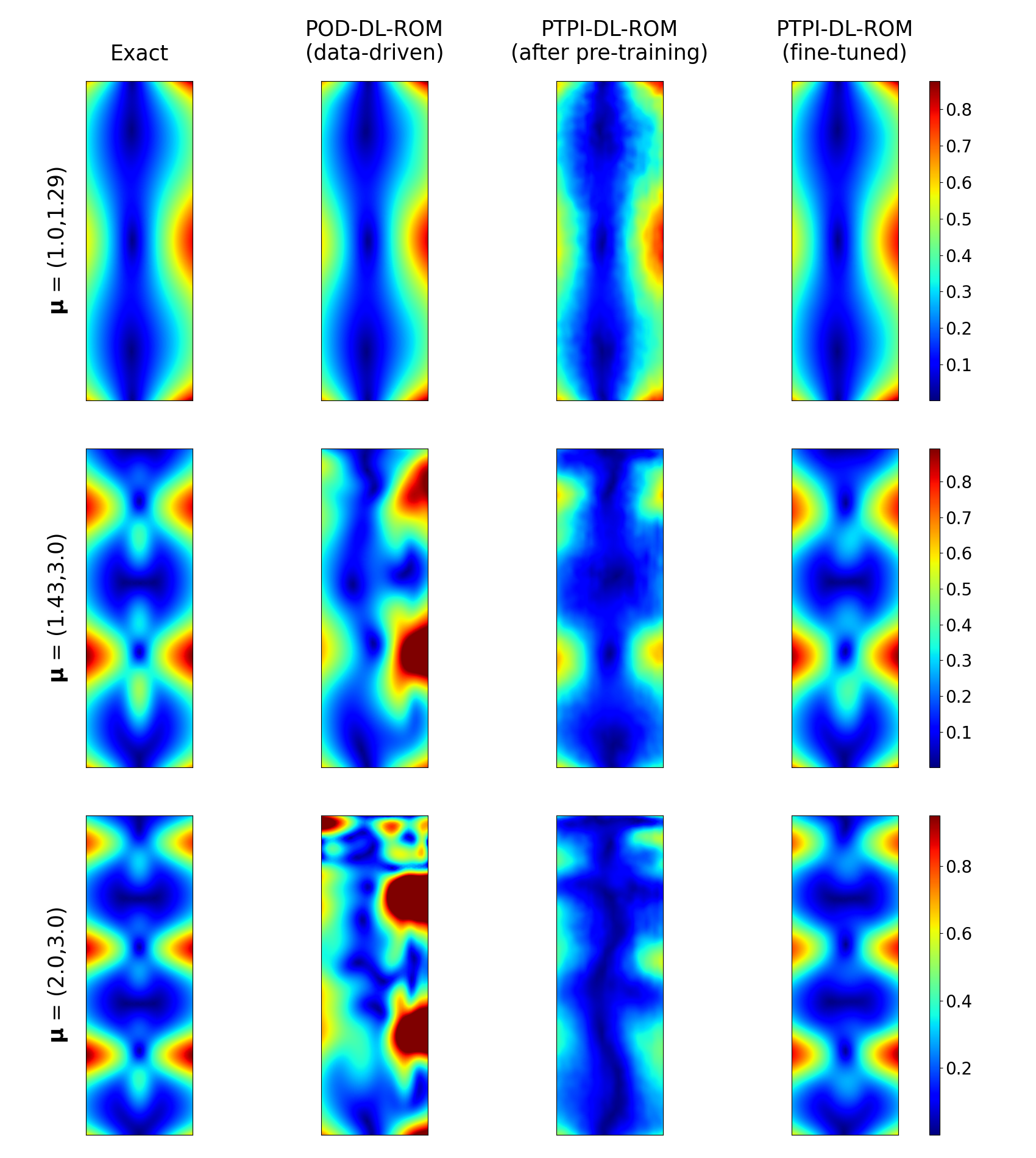}
    \caption{Darcy equations: visualization of the flux magnitude $\|\bm{\sigma}\|$ in the $yx$ plane for three different instances of the parameter $\muv$. We remark that the last two instances belong to the extrapolation regime.}
    \label{fig:darcy_flux}
\end{figure}

Our analysis is further supported by the visualization of the pressure field and the flux magnitude, reported in Fig.~\ref{fig:darcy_pres} and Fig.~\ref{fig:darcy_flux}, respectively. Indeed, we highlight that both POD-DL-ROMs and the fine-tuned PTPI-DL-ROMs are reliable in the interpolation regime (first row). On the contrary, in the extrapolation regime (second and third row) the fine-tuned PTPI-DL-ROM is extremely accurate, whereas the data-driven POD-DL-ROM's predictions are skewed.

As a final remark, we emphasize that the physics-based loss functional enables an evident spatial regularization of the pressure field, therefore enhancing the overall solution quality, compared to the high-fidelity solution; this latter is indeed less smooth since it relies on a  $\mathbb{P}^0$-DG discretization.

\subsection{Navier-Stokes Equations}

Within this last numerical experiment we tackle the case of parameter-dependent incompressible Navier-Stokes equations in a two-dimensional domain. Specifically, we deal with the simulation of the fluid flow induced by a prescribed inlet through a backward-facing step geometry. Here, we aim at demonstrating the capability of the proposed physics-informed strategy to handle a time extrapolation task in a nonlinear time-dependent vector problem featuring a more complex - and parameter-dependent - physics.

\smallskip
\noindent\textbf{Problem description.} We consider the formulation of a Navier-Stokes flow in its adimensional formulation, namely,
\begin{equation}
    \left\{
    \begin{aligned}
    \frac{\partial{\bm{u}}}{\partial{t}} + (\bm{u}\cdot \nabla) \bm{u} - \frac{1}{Re} \Delta \bm{u} + \nabla p &= 0 ,& \mbox{in }\Omega \times (0,T) \\
    \nabla \cdot \bm{u} &= 0,& \mbox{in } \Omega \times (0,T) \\
    \bm{u} &= \bm{u}_D(t) ,& \mbox{on }\Gamma_{IN} \times (0,T)\\
    \bm{u} &= \bm{0} ,& \mbox{on }\Gamma_{WALL} \times (0,T) \\
    \nabla u \mathbf{n} - p \mathbf{n} &= \bm{0} ,& \mbox{on }\Gamma_{OUT} \times (0,T) \\
    \bm{u}(t=0) &= \bm{0} ,& \mbox{in } \Omega,
    \end{aligned}
    \right.
\end{equation}
where  $\mathbf{n}$ denotes the outward normal unity vector and $Re>0$ the Reynolds number. The computational domain is $\Omega = ((0,4) \times (0,1)) \setminus ((0,1) \times (0,0.5))$. We define the inlet boundary as $\Gamma_{IN} = \{ (x,y) \in \partial{\Omega}: x = 0, y \in [0.5,1]\}$, the outlet boundary as $ \Gamma_{OUT} = \{ (x,y) \in \partial{\Omega}: x = 4, y \in [0,1] \}$, and  the internal wall $\Gamma_{WALL} = \partial{\Omega} \setminus (\Gamma_{IN} \cup \Gamma_{OUT})$. We force a time-dependent Dirichlet datum at the inlet, namely $\bm{u}_D(y;t) = [2 (1-e^{-3t}) (y - 0.5) (1 - y) , 0]$, while we employ a no-slip condition at the wall and endow the outlet with natural boundary conditions. We highlight that the problem is parametrized by means of the Reynolds number $Re$: therefore, for the sake of consistency with the previously employed notation, we set $\mu = Re$.

\smallskip
\noindent\textbf{Data collection and architectures.} We generate high-fidelity snapshots by employing an inf-sup stable couple of Finite Element spaces, namely $\mathbb{P}^1\textnormal{b}-\mathbb{P}^1$, once the problem is suitably discretized on a mesh of $N_h = 1722$ vertices. Thus, the semi-discrete formulation features a system of 8166 dofs. Then, we employ the Chorin-Temam splitting scheme with an explicit treatment of every left-hand-side term. We collect a total of $N_s = 10$ training samples on a time grid characterized by $\Delta t = 0.05$, yielding $\mathsf{P}_{sup} = \{200+ 100/3k, \mbox{ for } k=0,\ldots,N_s-1\} \subset \mathcal{P}_{sup} = [200,500]$ and $\mathsf{T}_{sup} = \{0.05k, \mbox{ for } k=1,\ldots,20\} \subset \mathcal{T}_{sup} = [0,1]$ respectively. Then, we generate $N_s^{test} = 5$ test samples for $\mathsf{P}_{test} = \{210 + 70k, \mbox{ for } k=0,\ldots,N_s^{test}-1\} \subset \mathcal{P}_{test} \equiv \mathcal{P}_{sup} = [200,500]$. The considered test time grid is $\mathsf{T}_{test} = \{0.05k, \mbox{ for } k=1,\ldots,80\} \subset \mathcal{T}_{test} = [0,4]$; thus, we are dealing with a time extrapolation task since $\mathcal{T}_{sup} \subsetneqq  \mathcal{T}_{test}$. We stress that we employ $N=20$ modes per channel for the preliminary dimensionality reduction and a latent dimension of $n=5$. The latter choice ensures a suitable preliminary compression of the supervised data and the possibility to adequately capture the variability of the solution manifold at the same time. 

We then design the neural network core as follows: {\em (i)} the Fourier-features-enhanced trunk net consists of 6 hidden layers of 100 neurons each, and is endowed with a SILU activation, {\em (ii)} the branch encoder is a dense network of 4 hidden layers of 50 units each with a ELU nonlinearity, {\em (iii)} the branch decoder is a composed of 9 hidden dense layers of 100 neurons each activated by the ELU function, and {\em (iv)} the branch reduced network consists of 4 hidden dense layers of 100 units each, and is endowed with the ELU activation.

\smallskip
\noindent\textbf{Training specifics.} 
We pre-train the trunk net for $3000$ epochs with a batch size of $30$ and a learning rate of $3 \times 10^{-4}$. The branch net with $5000$ gradient descent epochs using a learning rate of $3 \times 10^{-4}$, employing a batch size of $30$ for the labeled data and of $10$ for the unlabeled data, respectively. Finally, we fine-tune the resulting PTPI-DL-ROM for $5000$ epochs: to do that, we lower the learning rate to $10^{-4}$.
Since $\mathcal{P}_{sup} \times \mathcal{T}_{sup} \subsetneq \mathcal{P}_{test} \times \mathcal{T}_{test}$, the present test case encompasses an extrapolation task. 

Thus, we compensate the small data regime by enforcing the physics-based constraint in $\mathcal{P}_{res} \times \mathcal{T}_{res} = [200,500] \times [0,4.1] \supset \mathcal{P}_{test} \times \mathcal{T}_{test}$. Specifically, we draw $\mathsf{D}_{res} = \{(\mu,t)_j \sim \mathcal{U}(\mathcal{P}_{res} \times \mathcal{T}_{res})\}_{j=1}^{100}$ each $5$ gradient descent epochs. Moreover, we remark that with respect to the pre-training stage, we choose to compute the physics-based residual at the mesh points. Nonetheless, while fine-tuning we decide to enforce the boundary conditions at the boundary mesh points, whereas the PDE residual is computed for $\{\mathbf{y}_i \sim \mathcal{U}(\{\mathbf{x}_l\}_{l=1}^{N_h})\}_{i=1}^{500} \cup \{\mathbf{y}_i \in \mathcal{U}([1,2]\times[0,0.5])\}_{i=1}^{100} \cup \{\mathbf{y}_i \in \mathcal{U}([0,1]\times[0.5,1])\}_{i=1}^{200}$.
We note that the weighting criterion for the loss formulation terms is more complex than in the previous experiments, since a series of complex equations is involved. Specifically, while we employ $\omega_N=\omega_n=0.5$ for every output channel for the labeled data, $\omega_{\Omega} = [10,10,10]$ assumes a vectorial form, being the first and the second term related to the momentum equation of the $x$-velocity and the $y$-velocity, respectively, whereas the third term controls the incompressibility constraint. Similarly, we select $\omega_{\partial \Omega} = [100,1000,1,1]$, being the first two terms related to the Dirichlet boundary conditions, and the remaining to the natural boundary conditions. For the sake of comparison, we emphasize that we train a data-driven POD-DL-ROM for a total of $5000$ epochs employing a learning rate of $3 \times 10^{-4}$.

\begin{figure}[htb!]
    \centering
    \includegraphics[width=0.99\textwidth]{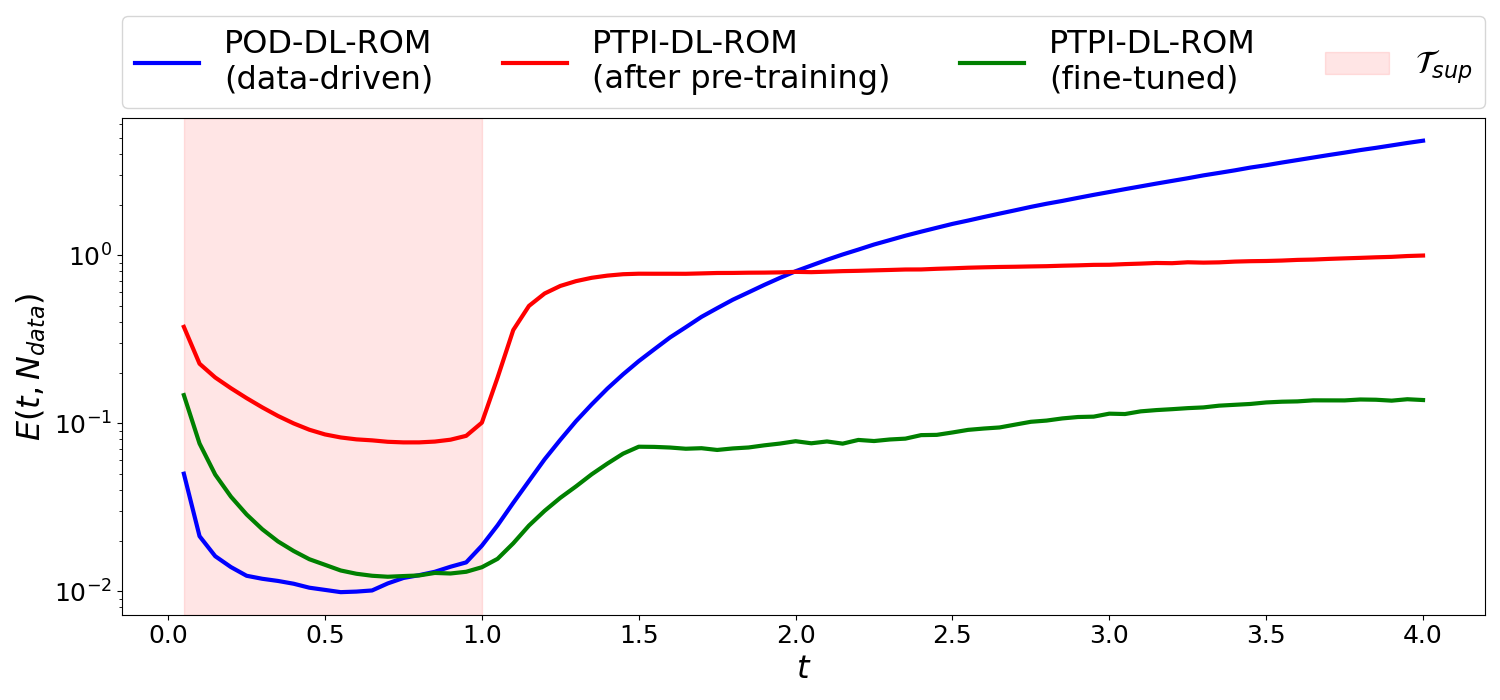}
    \caption{NS equations: visualization of $E(t,N_{data})$ for the comparison between data-driven POD-DL-ROM and the proposed PTPI-DL-ROM.}
    \label{fig:ns_err}
\end{figure}

\begin{figure}[htb!]
    \centering
    \includegraphics[width=0.99\textwidth]{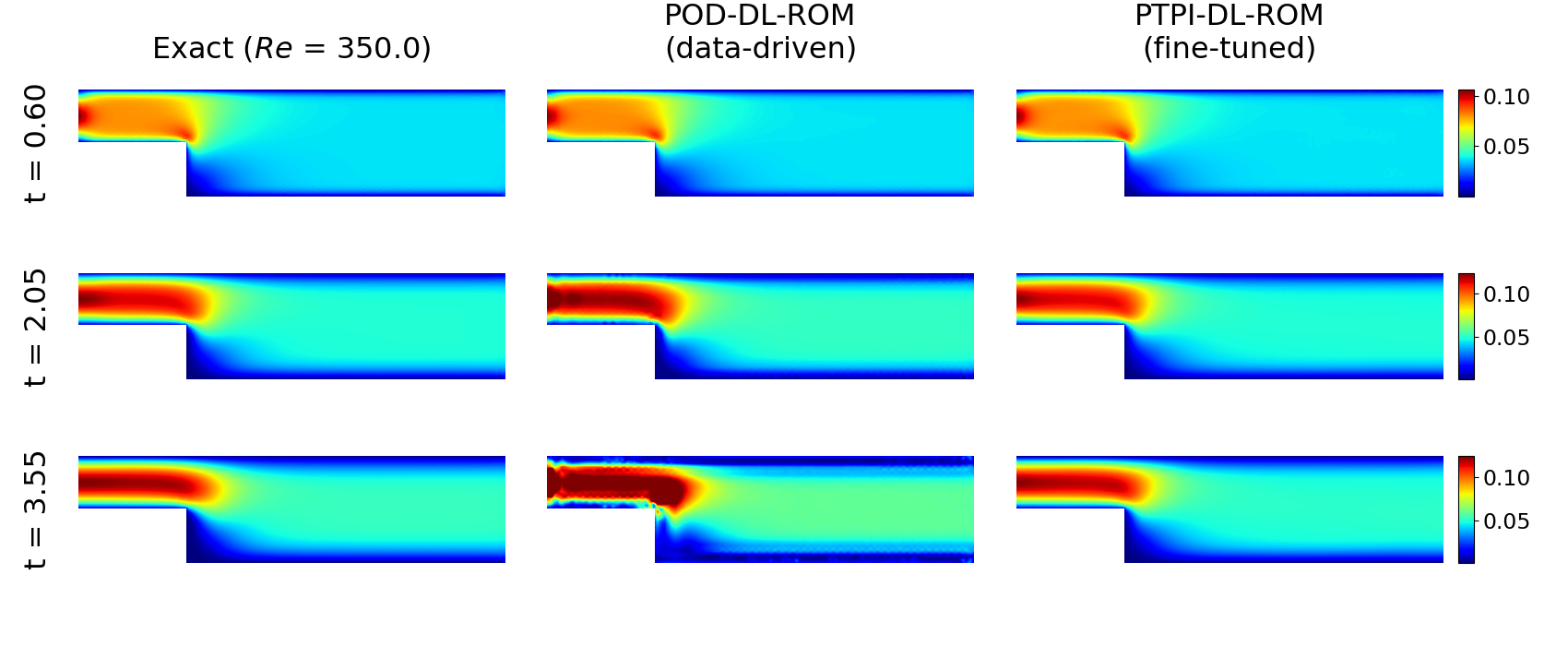}
    \caption{NS equations: comparison between the velocity magnitude accuracy of the data-driven POD-DL-ROM and fine-tuned PTPI-DL-ROM simulations at different time instances.}
    \label{fig:ns_visual_u}
\end{figure}

\begin{figure}[htb!]
    \centering
    \includegraphics[width=0.99\textwidth]{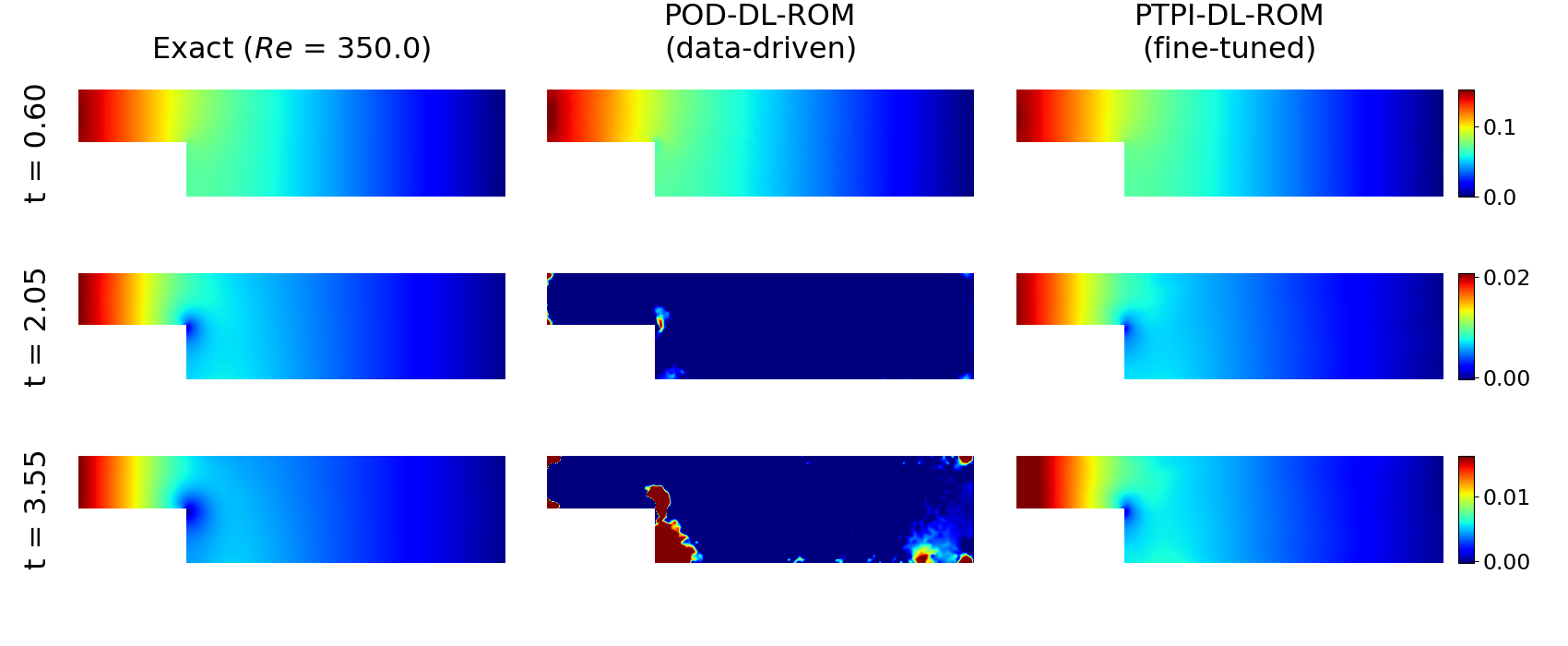}
    \caption{NS equations: comparison between the pressure accuracy of the data-driven POD-DL-ROM and fine-tuned PTPI-DL-ROM simulations at different time instances.}
    \label{fig:ns_visual_p}
\end{figure}

\smallskip
\noindent\textbf{Results analysis.} The present numerical experiment is employed for the validation of the time-extrapolation capabilities of PTPI-DL-ROM. In particular, we observe from Fig.~\ref{fig:ns_err} that the PTPI-DL-ROM showcases outstanding accuracy in the interpolation regime, i.e. for $t \in \mathcal{T}_{sup} = [0,T]$, where both data and physics are available. Nevertheless, we note that, even though the data-driven POD-DL-ROM is extremely accurate in the regions where labeled data are available, its reliability become skewed just beyond $\mathcal{T}_{sup}$ and the relative error steadily increases in the extrapolation regime. In contrast, the proposed PTPI-DL-ROM preserves a very good performance even for extrapolation samples, taking advantage of physics to compensate for the absence of supervised training data. Our analysis is further validated by Fig.~\ref{fig:ns_visual_u} and \ref{fig:ns_visual_p}, which emphasize the reliability of the PTPI-DL-ROM technique both in the interpolation regime (first row) and in the extrapolation regime (second and third row).

\section*{Conclusions}
The recent success of low-rank deep learning-based neural network architectures \cite{hesthaven2018,fresca2022poddlrom,lu2022comprehensive,lu2021learning} in the ROM field is due to their flexibility, their outstanding (and provable) approximation capabilities, as well as  their lightweight architectures. However, theoretical investigations and numerical experiments \cite{brivio2023error,bhattacharya2021,lanthaler2022error} involving these techniques have shown that their test accuracy strongly depends on the amount of available training data. This latter fact represents a huge drawback since very often the data generation phase is barely affordable because of the high computational burden entailed by the high-fidelity solvers, in terms of both computational time and resources, especially when dealing with complex problems.

Recent investigations  \cite{wang2021learning,safonova2023ten} have argued that labeled data and the underlying physical equations convey the same information, thus making possible to  learn parametric operators by minimizing the reconstruction error of the available training data and by enforcing a soft constraint based on the underlying physics in the regions of the time-parameter domain where training data are not accessible. 

Along this strand, we have proposed a novel framework based on a NN architecture that combines the best features of POD-DL-ROMs and DeepONets, two widely used low-rank deep learning-based architectures, to design a training strategy that significantly reduces the computational burden entailed by the computation of physics-based residual. Indeed, despite being extremely versatile, usually the computation of the strong continuous formulation of the residual through AD also becomes prohibitive: this further bottleneck is caused by the huge size of the computational graph entailed by AD, especially as the problem complexity increases and if the residual calculation requires the evaluation of high-order spatial derivatives. Thus, in order to mitigate such computational burden, and motivating our choices in light of the theory, we have devised a new affordable training algorithm consisting of {\em{(i)}} a lightweight pre-training phase that efficiently blends physics and data in order to approach an adequate minimum in the loss landscape, and {\em{(ii)}} a more expensive fine-tuning stage that enables the proposed architecture to reach a suitable test accuracy. Ultimately, we have assessed the extremely good performances of our proposed strategy on a wealth of numerical test cases involving different physical models and operators.

Several working directions may stem from the present work. For instance, in order to enhance the extrapolation capabilities in the time domain, we may consider suitable advanced architectures, such as LSTM \cite{fresca2023long} or Neural ODEs \cite{chen2018neural}. 
Moreover, in this work we focused on POD-based architectures to generate the linear trial manifold during the pre-training phase, whereas in the flourishing ROM literature other alternatives have emerged as well \cite{oleary2022}.
Nonetheless, we mention that recent developments in physics-informed machine learning approaches aim at proposing novel sampling techniques and hyperparameter tuning criteria \cite{daw2023mitigating,jin2021nsfnets,wang2022improved}, allowing for a more rapid convergence to suitable optima during the training phase. We believe that these new techniques will further enhance the training of PTPI-DL-ROMs and will enable us to tackle even more complex problems stemming from Applied Sciences and Engineering.

\section*{Acknowledgments}
The authors are members of the Gruppo Nazionale Calcolo Scientifico-Istituto Nazionale di Alta Matematica (GNCS-INdAM) and acknowledge the project “Dipartimento di Eccellenza” 2023-2027, funded by MUR, as well as the support of Fondazione Cariplo, Italy, Grant n. 2019-4608. AM acknowledges the PRIN 2022 Project “Numerical approximation of uncertainty quantification problems for PDEs by multi-fidelity methods (UQ-FLY)” (No. 202222PACR), funded by the European Union - NextGenerationEU, and the Project “Reduced Order Modeling and Deep Learning for the real- time approximation of PDEs (DREAM)” (Starting Grant No. FIS00003154), funded by the Italian Science Fund (FIS) - Ministero dell'Università e della Ricerca. AM and SF acknowledge the project FAIR (Future Artificial Intelligence Research), funded by the NextGenerationEU program within the PNRR-PE-AI scheme (M4C2, Investment 1.3, Line on Artificial Intelligence).

\appendix

\section{Additional proofs}

\subsection{Proof of Lemma \ref{lemma:pod-convergence}}
\begin{proof}
By definition, we have that
\begin{equation*}
    \int_{\mathcal{P}_{sup} \times \mathcal{T}_{sup}} \|\mathbf{u}_h(\muv, t) -  \mathbf{V}_{\infty}\mathbf{V}_{\infty}^T\mathbf{u}_h(\muv, t)\|^2  d(\muv,t) =  \sum_{k>N} \sigma_{k,\infty}^2,
\end{equation*}
so that
\begin{equation*}
\begin{aligned}
    &\mathcal{E}_{POD}^{gen}(\mathcal{P}_{sup} \times \mathcal{T}_{sup};N_s,N_t) = \\
    &\hspace{2cm}=\int_{\mathcal{P}_{sup} \times \mathcal{T}_{sup}} \|\mathbf{u}_h(\muv, t) - \mathbf{V}\mathbf{V}^T\mathbf{u}_h(\muv, t)\|^2 d(\muv,t) - \sum_{k>N} \sigma_k^2 + \sum_{k>N} \sigma_k^2 - \sum_{k>N} \sigma_{k,\infty}^2 \\
    &\hspace{2cm}\le  \biggl| \int_{\mathcal{P}_{sup} \times \mathcal{T}_{sup}} \|\mathbf{u}_h(\muv, t) - \mathbf{V}\mathbf{V}^T\mathbf{u}_h(\muv, t)\|^2 d(\muv,t) - \sum_{k>N} \sigma_k^2 \biggr| + \biggl|\sum_{k>N} \sigma_k^2 - \sum_{k>N} \sigma_{k,\infty}^2\biggr|.
\end{aligned}
\end{equation*}
We observe that, by definition of $\sum_{k>N} \sigma_k^2$, 
\begin{equation*}
\begin{aligned}
    &\biggl|\int_{\mathcal{P}_{sup} \times \mathcal{T}_{sup}} \|\mathbf{u}_h(\muv, t) - \mathbf{V}\mathbf{V}^T\mathbf{u}_h(\muv, t)\|^2  d(\muv,t) - \sum_{k>N} \sigma_k^2 \biggr| = \\
    &\hspace{2cm} = \biggl|\int_{\mathcal{P}_{sup} \times \mathcal{T}_{sup}} \|\mathbf{u}_h(\muv, t) - \mathbf{V}\mathbf{V}^T\mathbf{u}_h(\muv, t)\|^2  d(\muv,t) -  \\
    &\hspace{5cm} - \frac{|\mathcal{P}_{sup} \times \mathcal{T}_{sup}|}{N_sN_t} \sum_{(\muv,t) \in \mathsf{P}_{sup} \times \mathsf{T}_{sup}}\|\mathbf{u}_h(\muv,t) - \mathbf{V}\mathbf{V}^T\mathbf{u}_h(\muv,t)\|^2 \biggr| \rightarrow 0
\end{aligned}
\end{equation*}
as $N_s, N_t \rightarrow \infty$, by choosing an adequate sampling strategy for $\mathsf{P}_{sup} \times \mathsf{T}_{sup}$ (for instance, $\muv \sim \mathcal{U}(\mathcal{P}_{sup})$ and $t \in \{k\mathcal{T}_{sup}/N_t\}_{k=1}^{N_t}$: in this case the convergence is a.s., see Proposition 1 of \cite{brivio2023error}). 
Then, thanks again to the results of Section 2 of \cite{brivio2023error}, we have that:
\begin{equation*}
    \sum_{k>N} \sigma_k^2 \xrightarrow{a.s.} \sum_{k>N} \sigma_{k,\infty}^2, \qquad N_s, N_t \rightarrow \infty,
\end{equation*}
which concludes the proof.
\end{proof}

\bibliographystyle{plain} 
\bibliography{biblio}

\begin{thebibliography}{10}

\bibitem{baydin2017automatic}
At\i{}l\i{}m~G\"{u}nes Baydin, Barak~A. Pearlmutter, Alexey~Andreyevich Radul,
  and Jeffrey~Mark Siskind.
\newblock Automatic differentiation in machine learning: a survey.
\newblock {\em J. Mach. Learn. Res.}, 18(1):5595–5637, jan 2017.

\bibitem{benner2020model_1}
Peter Benner, Stefano Grivet-Talocia, Alfio Quarteroni, Gianluigi Rozza, Wil
  Schilders, and Lu{\'\i}s~Miguel Silveira~(Eds.).
\newblock {\em Model Order Reduction, Volume 1: System-and Data-Driven Methods
  and Algorithms}.
\newblock De Gruyter, Berlin/Boston, 2021.

\bibitem{benner2017model}
Peter Benner, Mario Ohlberger, Albert Cohen, and Karen Willcox~(Eds.).
\newblock {\em Model reduction and approximation: theory and algorithms}.
\newblock SIAM, Philadephia, PA, 2017.

\bibitem{benner2020model_2}
Peter Benner, Wil Schilders, Stefano Grivet-Talocia, Alfio Quarteroni,
  Gianluigi Rozza, and Lu{\'\i}s Miguel Silveira~(Eds.).
\newblock {\em Model Order Reduction, Volume 2: Snapshot-Based Methods and
  Algorithms}.
\newblock De Gruyter, Berlin/Boston, 2020.

\bibitem{bhattacharya2021}
Kaushik Bhattacharya, Bamdad Hosseini, Nikola~B Kovachki, and Andrew~M Stuart.
\newblock Model reduction and neural networks for parametric pdes.
\newblock {\em SMAI J. Comput. Mat.}, 7:121--157, 2021.

\bibitem{brigato2021close}
Lorenzo Brigato and Luca Iocchi.
\newblock A close look at deep learning with small data.
\newblock In {\em 2020 25th International Conference on Pattern Recognition
  (ICPR)}, pages 2490--2497, Los Alamitos, CA, USA, jan 2021. IEEE Computer
  Society.

\bibitem{brivio2023error}
Simone Brivio, Stefania Fresca, Nicola~Rares Franco, and Andrea Manzoni.
\newblock Error estimates for {P}{O}{D}-{D}{L}-{R}{O}{M}s: a deep learning
  framework for reduced order modeling of nonlinear parametrized {P}{D}{E}s
  enhanced by proper orthogonal decomposition.
\newblock {\em Adv. Comput. Math.}, 50(33), 2024.

\bibitem{buithanh2003proper}
Tan Bui-Thanh, Murali Damodaran, and Karen Willcox.
\newblock Proper orthogonal decomposition extensions for parametric
  applications in compressible aerodynamics.
\newblock In {\em 21st AIAA Applied Aerodynamics Conference}, 06 2003.

\bibitem{chen2018neural}
Ricky T.~Q. Chen, Yulia Rubanova, Jesse Bettencourt, and David~K Duvenaud.
\newblock Neural ordinary differential equations.
\newblock In S.~Bengio, H.~Wallach, H.~Larochelle, K.~Grauman, N.~Cesa-Bianchi,
  and R.~Garnett, editors, {\em Advances in Neural Information Processing
  Systems}, volume~31. Curran Associates, Inc., 2018.

\bibitem{chen2021}
Wenqian Chen, Qian Wang, Jan~S. Hesthaven, and Chuhua Zhang.
\newblock Physics-informed machine learning for reduced-order modeling of
  nonlinear problems.
\newblock {\em J. Comp. Phys.}, 446:110666, 2021.

\bibitem{daw2023mitigating}
Arka Daw, Jie Bu, Sifan Wang, Paris Perdikaris, and Anuj Karpatne.
\newblock Mitigating propagation failures in physics-informed neural networks
  using retain-resample-release (r3) sampling, 2023.

\bibitem{drakoulas2023fastsvd}
George Drakoulas, Theodore Gortsas, George~C. Bourantas, Vasilis~N. Burganos,
  and Demosthenes Polyzos.
\newblock Fast{S}{V}{D}-{M}{L}-{R}{O}{M}: A reduced-order modeling framework
  based on machine learning for real-time applications.
\newblock {\em Comput. Methods Appl. Mech. Engrg.}, 414:116155, 2023.

\bibitem{eckart1936}
Carl Eckart and G.~Marion Young.
\newblock The approximation of one matrix by another of lower rank.
\newblock {\em Psychometrika}, 1:211--218, 1936.

\bibitem{ernst2022certified}
Lewin Ernst and Karsten Urban.
\newblock A certified wavelet-based physics-informed neural network for the
  solution of parameterized partial differential equations, 2022.

\bibitem{farhat2020}
Charbel Farhat, Sebastian Grimberg, Andrea Manzoni, and Alfio Quarteroni.
\newblock Computational bottlenecks for proms: precomputation and
  hyperreduction.
\newblock In Peter Benner, Stefano Grivet-Talocia, Alfio Quarteroni, Gianluigi
  Rozza, Wil Schilders, and Luís~Miguel Silveira, editors, {\em Volume 2:
  Snapshot-Based Methods and Algorithms}, pages 181--244. De Gruyter, Berlin,
  Boston, 2020.

\bibitem{franco2023deep}
Nicola~R. Franco, Andrea Manzoni, and Paolo Zunino.
\newblock A deep learning approach to reduced order modelling of parameter
  dependent partial differential equations.
\newblock {\em Math. Comp.}, 92(340):483--524, 2023.

\bibitem{fresca2021dlrom}
Stefania Fresca, Luca Dede, and Andrea Manzoni.
\newblock A comprehensive deep learning-based approach to reduced order
  modeling of nonlinear time-dependent parametrized pdes.
\newblock {\em J. Sci. Comput.}, 87(2):1--36, 2021.

\bibitem{fresca2023long}
Stefania Fresca, Federico Fatone, and Andrea Manzoni.
\newblock Long-time prediction of nonlinear parametrized dynamical systems by
  deep learning-based reduced order models.
\newblock {\em Math. Eng.}, 5(6):1--36, 2023.

\bibitem{fresca2022microstructures}
Stefania Fresca, Giorgio Gobat, Patrick Fedeli, Attilio Frangi, and Andrea
  Manzoni.
\newblock Deep learning-based reduced order models for the real-time simulation
  of the nonlinear dynamics of microstructures.
\newblock {\em International Journal for Numerical Methods in Engineering},
  123(20):4749--4777, 2022.

\bibitem{fresca2021fluids}
Stefania Fresca and Andrea Manzoni.
\newblock Real-time simulation of parameter-dependent fluid flows through deep
  learning-based reduced order models.
\newblock {\em Fluids}, 6(7), 2021.

\bibitem{fresca2022poddlrom}
Stefania Fresca and Andrea Manzoni.
\newblock {P}{O}{D}-{D}{L}-{R}{O}{M}: Enhancing deep learning-based reduced
  order models for nonlinear parametrized pdes by proper orthogonal
  decomposition.
\newblock {\em Comput. Methods Appl. Mech. Engrg.}, 388:114181, 2022.

\bibitem{fresca2021electrophys}
Stefania Fresca, Andrea Manzoni, Luca Ded{\`e}, and Alfio Quarteroni.
\newblock {P}{O}{D}-enhanced deep learning-based reduced order models for the
  real-time simulation of cardiac electrophysiology in the left atrium.
\newblock {\em Front. Physiol.}, page 1431, 2021.

\bibitem{gonzalez2018deep}
Francisco~J Gonzalez and Maciej Balajewicz.
\newblock Deep convolutional recurrent autoencoders for learning
  low-dimensional feature dynamics of fluid systems.
\newblock {\em arXiv preprint arXiv:1808.01346}, 2018.

\bibitem{goodfellow2016deep}
Ian Goodfellow, Yoshua Bengio, and Aaron Courville.
\newblock {\em Deep learning}.
\newblock MIT press, 2016.

\bibitem{gopakumar2023loss}
Vignesh Gopakumar, Stanislas Pamela, and Debasmita Samaddar.
\newblock Loss landscape engineering via data regulation on pinns.
\newblock {\em Machine Learning with Applications}, 12:100464, 2023.

\bibitem{haghighat2024endeeponet}
Ehsan Haghighat, Umair bin Waheed, and George Karniadakis.
\newblock En-deeponet: An enrichment approach for enhancing the expressivity of
  neural operators with applications to seismology.
\newblock {\em Comput. Methods Appl. Mech. Engrg.}, 420:116681, 2024.

\bibitem{han2018solving}
Jiequn Han, Arnulf Jentzen, and Weinan E.
\newblock Solving high-dimensional partial differential equations using deep
  learning.
\newblock {\em Proc. Nat. Acad. Sci.}, 115(34):8505--8510, 2018.

\bibitem{hernandez2021deep}
Quercus Hern{\'a}ndez, Alberto Badias, David Gonzalez, Francisco Chinesta, and
  Elias Cueto.
\newblock Deep learning of thermodynamics-aware reduced-order models from data.
\newblock {\em Comput. Methods Appl. Mech. Engrg.}, 379:113763, 2021.

\bibitem{hernandez2021structure}
Quercus Hern{\'a}ndez, Alberto Bad{\'\i}as, David Gonz{\'a}lez, Francisco
  Chinesta, and El{\'\i}as Cueto.
\newblock Structure-preserving neural networks.
\newblock {\em J. Comput. Phys.}, 426:109950, 2021.

\bibitem{hesthaven2018}
Jan~S. Hesthaven and Stefano Ubbiali.
\newblock Non-intrusive reduced order modeling of nonlinear problems using
  neural networks.
\newblock {\em J. Comp. Phys.}, 363:55--78, 2018.

\bibitem{jin2021nsfnets}
Xiaowei Jin, Shengze Cai, Hui Li, and George~Em Karniadakis.
\newblock Nsfnets (navier-stokes flow nets): Physics-informed neural networks
  for the incompressible navier-stokes equations.
\newblock {\em J. Comp. Phys.}, 426:109951, 2021.

\bibitem{karniadakis2021physics}
George~Em Karniadakis, Ioannis~G Kevrekidis, Lu~Lu, Paris Perdikaris, Sifan
  Wang, and Liu Yang.
\newblock Physics-informed machine learning.
\newblock {\em Nature Rev. Phys.}, 3(6):422--440, 2021.

\bibitem{kingma2017adam}
Diederik~P. Kingma and Jimmy Ba.
\newblock Adam: A method for stochastic optimization, 2017.

\bibitem{kovachi2023neural}
Nikola Kovachki, Zongyi Li, Burigede Liu, Kamyar Azizzadenesheli, Kaushik
  Bhattacharya, Andrew Stuart, and Anima Anandkumar.
\newblock Neural operator: Learning maps between function spaces with
  applications to pdes.
\newblock {\em Journal of Machine Learning Research}, 24(89):1--97, 2023.

\bibitem{krishnapriyan2021characterizing}
Aditi~S. Krishnapriyan, Amir Gholami, Shandian Zhe, Robert~M. Kirby, and
  Michael~W. Mahoney.
\newblock Characterizing possible failure modes in physics-informed neural
  networks, 2021.

\bibitem{kumar2024combustion}
Anuj Kumar and Tarek Echekki.
\newblock Combustion chemistry acceleration with deeponets.
\newblock {\em Fuel}, 365:131212, 2024.

\bibitem{lanthaler2022error}
Samuel Lanthaler, Siddhartha Mishra, and George~Em Karniadakis.
\newblock Error estimates for deeponets: a deep learning framework in infinite
  dimensions.
\newblock {\em Transactions of Mathematics and Its Applications}, 6(1):tnac001,
  2022.

\bibitem{lee2020model}
Kookjin Lee and Kevin~T Carlberg.
\newblock Model reduction of dynamical systems on nonlinear manifolds using
  deep convolutional autoencoders.
\newblock {\em J. Comput. Phys.}, 404:108973, 2020.

\bibitem{li2020fourier}
Zongyi Li, Nikola Kovachki, Kamyar Azizzadenesheli, Burigede Liu, Kaushik
  Bhattacharya, Andrew Stuart, and Anima Anandkumar.
\newblock Fourier neural operator for parametric partial differential
  equations, 2020.

\bibitem{lu2021learning}
Lu~Lu, Pengzhan Jin, Guofei Pang, Zhongqiang Zhang, and George~Em Karniadakis.
\newblock Learning nonlinear operators via {DeepONet} based on the universal
  approximation theorem of operators.
\newblock {\em Nature Machine Intelligence}, 3(3):218--229, 2021.

\bibitem{lu2022comprehensive}
Lu~Lu, Xuhui Meng, Shengze Cai, Zhiping Mao, Somdatta Goswami, Zhongqiang
  Zhang, and George~Em Karniadakis.
\newblock A comprehensive and fair comparison of two neural operators (with
  practical extensions) based on fair data.
\newblock {\em Comput. Methods Appl. Mech. Engrg.}, 393:114778, 2022.

\bibitem{mckay2023comparison}
Michael~D. Mckay, Richard~J. Beckman, and William~J. Conover.
\newblock A comparison of three methods for selecting values of input variables
  in the analysis of output from a computer code.
\newblock {\em Technometrics}, 21(2):239--245, 1979.

\bibitem{mucke2021}
Nikolaj~T. Mücke, Sander~M. Bohté, and Cornelis~W. Oosterlee.
\newblock Reduced order modeling for parameterized time-dependent pdes using
  spatially and memory aware deep learning.
\newblock {\em J. Comput. Sci.}, 53:101408, 2021.

\bibitem{negri2016redbkit}
Federico Negri.
\newblock redb{KIT} {V}ersion 2.2.
\newblock \url{http://redbkit.github.io/redbKIT/}, 2016.

\bibitem{oleary2022}
Thomas O’Leary-Roseberry, Xiaosong Du, Anirban Chaudhuri, Joaquim~R.R.A.
  Martins, Karen Willcox, and Omar Ghattas.
\newblock Learning high-dimensional parametric maps via reduced basis adaptive
  residual networks.
\newblock {\em Comput. Methods Appl. Mech. Engrg.}, 402:115730, 2022.

\bibitem{pant2021}
Pranshu Pant, Ruchit Doshi, Pranav Bahl, and Amir~Barati Farimani.
\newblock Deep learning for reduced order modelling and efficient temporal
  evolution of fluid simulations.
\newblock {\em Phys. Fluids}, 33(10):107101, 2021.

\bibitem{quarteroni2017}
Alfio Quarteroni.
\newblock {\em Numerical Models for Differential Problems}.
\newblock Springer Cham, 2017.

\bibitem{QMN_16}
Alfio Quarteroni, Andrea Manzoni, and Federico Negri.
\newblock {\em Reduced Basis Methods for Partial Differential Equations. An
  Introduction}, volume~92 of {\em Unitext}.
\newblock Springer, 2016.

\bibitem{raissi2019physics}
Maziar Raissi, Paris Perdikaris, and George~Em Karniadakis.
\newblock Physics-informed neural networks: A deep learning framework for
  solving forward and inverse problems involving nonlinear partial differential
  equations.
\newblock {\em J. Comput. Phys.}, 378:686--707, 2019.

\bibitem{rohrhofer2023role}
{Franz Martin} Rohrhofer, Stefan Posch, Clemens G{\"o}{\ss}nitzer, and Bernhard
  Geiger.
\newblock On the role of fixed points of dynamical systems in training
  physics-informed neural networks.
\newblock {\em Transactions on Machine Learning Research}, 2023(1), January
  2023.

\bibitem{deryck2023operator}
Tim~De Ryck, Florent Bonnet, Siddhartha Mishra, and Emmanuel de~Bézenac.
\newblock An operator preconditioning perspective on training in
  physics-informed machine learning, 2023.

\bibitem{deryck2022wpinns}
Tim~De Ryck, Siddhartha Mishra, and Roberto Molinaro.
\newblock wpinns: Weak physics informed neural networks for approximating
  entropy solutions of hyperbolic conservation laws, 2022.

\bibitem{safonova2023ten}
Anastasiia Safonova, Gohar Ghazaryan, Stefan Stiller, Magdalena Main-Knorn,
  Claas Nendel, and Masahiro Ryo.
\newblock Ten deep learning techniques to address small data problems with
  remote sensing.
\newblock {\em International Journal of Applied Earth Observation and
  Geoinformation}, 125:103569, 2023.

\bibitem{san2019artificial}
Omer San, Romit Maulik, and Mansoor Ahmed.
\newblock An artificial neural network framework for reduced order modeling of
  transient flows.
\newblock {\em Comm. Nonlin. Sci. Numer. Simul.}, 77:271--287, 2019.

\bibitem{sharma2023stiff}
Prakhar Sharma, Llion Evans, Michelle Tindall, and Perumal Nithiarasu.
\newblock {Stiff-PDEs} and {Physics-Informed} neural networks.
\newblock {\em Archives of Computational Methods in Engineering},
  30(5):2929--2958, June 2023.

\bibitem{szlam2014}
Arthur Szlam, Yuval Kluger, and Mark Tygert.
\newblock An implementation of a randomized algorithm for principal component
  analysis, 2014.

\bibitem{wang2022l2}
Chuwei Wang, Shanda Li, Di~He, and Liwei Wang.
\newblock Is $l^2$ physics-informed loss always suitable for training
  physics-informed neural network?, 2022.

\bibitem{wang2019non}
Qian Wang, Jan~S. Hesthaven, and Deep Ray.
\newblock {Non-intrusive reduced order modeling of unsteady flows using
  artificial neural networks with application to a combustion problem}.
\newblock {\em J. Comput. Phys.}, 384:289--307, 2019.

\bibitem{wang2020recurrent}
Qian Wang, Nicolò Ripamonti, and Jan~S. Hesthaven.
\newblock Recurrent neural network closure of parametric {P}{O}{D}-galerkin
  reduced-order models based on the mori-zwanzig formalism.
\newblock {\em J. Comput. Phys.}, 410:109402, 2020.

\bibitem{wang2021learning}
Sifan Wang, Hanwen Wang, and Paris Perdikaris.
\newblock Learning the solution operator of parametric partial differential
  equations with physics-informed deeponets.
\newblock {\em Science Adv.}, 7(40):eabi8605, 2021.

\bibitem{wang2022improved}
Sifan Wang, Hanwen Wang, and Paris Perdikaris.
\newblock Improved architectures and training algorithms for deep operator
  networks.
\newblock {\em J. Sci. Comput.}, 92(2):35, June 2022.

\bibitem{wang2022pinns}
Sifan Wang, Xinling Yu, and Paris Perdikaris.
\newblock When and why pinns fail to train: A neural tangent kernel
  perspective.
\newblock {\em J. Comp. Phys.}, 449:110768, 2022.

\bibitem{willcox2002balanced}
Karen Willcox and Jaime Peraire.
\newblock Balanced model reduction via the proper orthogonal decomposition.
\newblock {\em Aiaa Journal - AIAA J}, 40:2323--2330, 11 2002.

\bibitem{xu2023small}
Pengcheng Xu, Xiaobo Ji, Minjie Li, and Wencong Lu.
\newblock Small data machine learning in materials science.
\newblock {\em npj Computational Materials}, 9(1):42, March 2023.

\bibitem{yang2021b}
Liu Yang, Xuhui Meng, and George~Em Karniadakis.
\newblock B-pinns: Bayesian physics-informed neural networks for forward and
  inverse pde problems with noisy data.
\newblock {\em J. Comput. Phys.}, 425:109913, 2021.

\bibitem{yang2024datadriven}
Sunwoong Yang, Hojin Kim, Yoonpyo Hong, Kwanjung Yee, Romit Maulik, and Namwoo
  Kang.
\newblock Data-driven physics-informed neural networks: A digital twin
  perspective, 2024.

\bibitem{zhu2023reliable}
Min Zhu, Handi Zhang, Anran Jiao, George~Em Karniadakis, and Lu~Lu.
\newblock Reliable extrapolation of deep neural operators informed by physics
  or sparse observations.
\newblock {\em Comput. Methods Appl. Mech. Engrg.}, 412:116064, 2023.

\end{thebibliography}
%\appto{\bibsetup}{\sloppy}

%\printbibliography

\end{document}